\documentclass[11pt,reqno]{amsart}
\usepackage[letterpaper,margin=1in,footskip=0.25in]{geometry}
\usepackage{microtype}
\DeclareMathSymbol{\smallsetminus}{\mathbin}{AMSb}{"72}
\usepackage[only,llbracket,rrbracket]{stmaryrd}
\usepackage{mathtools}
\usepackage{enumitem}
\usepackage[noadjust]{cite}
\renewcommand{\citepunct}{;\penalty\citemidpenalty\ }

\PassOptionsToPackage{dvipsnames}{xcolor}
\usepackage{tikz-cd}
\PassOptionsToPackage{pdfusetitle,colorlinks,linktocpage,pagebackref}{hyperref}
\usepackage[depth=2]{bookmark}
\hypersetup{citecolor=OliveGreen,linkcolor=Mahogany,urlcolor=Plum}
\usepackage[hyphenbreaks]{breakurl}

\newtheorem{alphthm}{Theorem}

\newtheorem{theorem}{Theorem}[section]
\newtheorem{lemma}[theorem]{Lemma}
\newtheorem{proposition}[theorem]{Proposition}
\newtheorem{corollary}[theorem]{Corollary}

\theoremstyle{definition}
\newtheorem{definition}[theorem]{Definition}
\newtheorem{example}[theorem]{Example}

\theoremstyle{remark}
\newtheorem{remark}[theorem]{Remark}

\newtheoremstyle{step}{.25\baselineskip\@plus.1\baselineskip\@minus.1\baselineskip}{.25\baselineskip\@plus.1\baselineskip\@minus.1\baselineskip}{\itshape}{}{\bfseries}{\bfseries .}{5pt plus 1pt minus 1pt}{\thmname{#1}\thmnumber{ #2}\thmnote{ \normalfont(#3)}}
\theoremstyle{step}
\newtheorem{step}{Step}[theorem]

\newtheoremstyle{cited}{.5\baselineskip\@plus.2\baselineskip\@minus.2\baselineskip}{.5\baselineskip\@plus.2\baselineskip\@minus.2\baselineskip}{\itshape}{}{\bfseries}{\bfseries .}{5pt plus 1pt minus 1pt}{\thmname{#1}\thmnumber{ #2}\thmnote{ \normalfont#3}}
\theoremstyle{cited}
\newtheorem{citedthm}[theorem]{Theorem}
\newtheorem{citedconj}[theorem]{Conjecture}
\newtheorem{citedlem}[theorem]{Lemma}
\newtheorem{citedprop}[theorem]{Proposition}

\newtheoremstyle{citeddef}{.5\baselineskip\@plus.2\baselineskip\@minus.2\baselineskip}{.5\baselineskip\@plus.2\baselineskip\@minus.2\baselineskip}{}{}{\bfseries}{\bfseries .}{5pt plus 1pt minus 1pt}{\thmname{#1}\thmnumber{ #2}\thmnote{ \normalfont#3}}
\theoremstyle{citeddef}
\newtheorem{citeddef}[theorem]{Definition}
\newtheorem{citedconstr}[theorem]{Construction}

\DeclareMathOperator{\Bs}{Bs}
\DeclareMathOperator{\Bminus}{\mathcode`\-=8704 \mathbf{B}_-}

\DeclareMathOperator{\Char}{char}
\DeclareMathOperator{\Div}{Div}
\DeclareMathOperator{\Ext}{Ext}
\DeclareMathOperator{\Hom}{Hom}
\DeclareMathOperator{\NSFR}{NSFR}
\DeclareMathOperator{\NNef}{NNef}
\DeclareMathOperator{\Nklt}{Nklt}
\DeclareMathOperator{\SB}{\mathbf{B}}
\DeclareMathOperator{\Spec}{Spec}
\DeclareMathOperator{\Supp}{Supp}
\DeclareMathOperator{\WDiv}{WDiv}
\DeclareMathOperator{\codim}{codim}

\DeclareMathOperator{\eval}{eval}
\DeclareMathOperator{\im}{im}
\DeclareMathOperator{\length}{length}

\renewcommand{\AA}{\mathbf{A}}

\newcommand{\DD}{\mathbf{D}}
\newcommand{\FF}{\mathbf{F}}
\newcommand{\HH}{\mathbf{H}}
\newcommand{\PP}{\mathbf{P}}
\newcommand{\QQ}{\mathbf{Q}}
\newcommand{\RR}{\mathbf{R}}
\newcommand{\ZZ}{\mathbf{Z}}
\newcommand{\kk}{\mathbf{k}}
\newcommand{\cJ}{\mathcal{J}}
\newcommand{\cO}{\mathcal{O}}
\newcommand{\cP}{\mathcal{P}}
\newcommand{\cQ}{\mathcal{Q}}
\newcommand{\fa}{\mathfrak{a}}
\newcommand{\fb}{\mathfrak{b}}
\newcommand{\fm}{\mathfrak{m}}
\newcommand{\fn}{\mathfrak{n}}
\newcommand{\fp}{\mathfrak{p}}
\newcommand{\Tr}{\mathrm{Tr}}
\newcommand{\num}{\mathrm{num}}
\newcommand{\perf}{\mathrm{perf}}
\newcommand{\qc}{\mathrm{qc}}
\newcommand{\red}{\mathrm{red}}

\begin{document}
\title[The gamma construction and asymptotic invariants of line bundles]{The
gamma construction and asymptotic invariants\\of line bundles over arbitrary
fields}
\author{Takumi Murayama}
\thanks{This material is based upon work supported by the National Science
Foundation under Grant Nos.\ DMS-1501461 and DMS-1701622}

\address{Department of Mathematics\\University of Michigan\\
Ann Arbor, MI 48109-1043\\USA}
\email{\href{mailto:takumim@umich.edu}{takumim@umich.edu}}
\urladdr{\url{http://www-personal.umich.edu/~takumim/}}

\makeatletter
  \hypersetup{
    pdftitle=The gamma construction and asymptotic invariants of line bundles
    over arbitrary fields,
    pdfsubject=\@subjclass,pdfkeywords=\@keywords
  }
\makeatother

\begin{abstract}
  We extend results on asymptotic invariants of line bundles on complex
  projective varieties to projective varieties over arbitrary fields.
  To do so over imperfect fields, we prove a scheme-theoretic version of the
  gamma construction of Hochster and Huneke to reduce to the setting where the
  ground field is $F$-finite.
  Our main result uses the gamma construction to extend the ampleness criterion
  of de Fernex, K\"uronya, and Lazarsfeld using asymptotic cohomological
  functions to projective varieties over arbitrary fields, which was previously
  known only for complex projective varieties.
  We also extend Nakayama's description of the restricted base locus to
  klt or strongly $F$-regular varieties over arbitrary fields.
\end{abstract}

\maketitle

\section{Introduction}\label{sect:intro}
Let $X$ be a projective variety over a field $k$.
When $k$ is the field of complex numbers, results from the minimal
model program can be used to understand the birational geometry of $X$.
When $k$ is an arbitrary field, however, many of these results and the
tools used to prove them are unavailable.
The most problematic situation is when $k$ is an imperfect field of
characteristic $p > 0$, in which case there are three major difficulties.
To begin with, since $k$ is of characteristic $p > 0$,
\begin{enumerate}[label=(\Roman*),ref=\Roman*]
  \item\label{problem:nores}
    Resolutions of singularities are not known to exist (see \cite{Hau10}), and
  \item\label{problem:novanishing}
    Vanishing theorems are false (Raynaud \cite{Ray78}).
\end{enumerate}
A common workaround for (\ref{problem:nores}) is to use de Jong's theory of
alterations \cite{dJ96}.
Circumventing (\ref{problem:novanishing}), on the other hand, is more difficult.
One useful approach is to exploit the Frobenius morphism $F\colon X \to X$ and
its Grothendieck trace $F_*\omega_X^\bullet \to \omega_X^\bullet$; see
\cite{PST17}.
For imperfect fields, however, this approach runs into another problem:
\begin{enumerate}[label=(\Roman*),ref=\Roman*,start=3]
  \item\label{problem:nofrobtech}
    Most applications of Frobenius techniques require $k$ to be
    \textsl{$F$-finite,} i.e., satisfy $[k:k^p] < \infty$. 
\end{enumerate}
This last issue arises since Grothendieck duality cannot be applied to the
Frobenius if it is not finite.
Recent advances in the minimal model program over imperfect fields due to Tanaka
\cite{Tan18,Tan} suggest that it would be worthwhile to develop
a systematic way to deal with (\ref{problem:nofrobtech}).
\medskip
\par Our first goal is to provide such a systematic way to reduce
to the case when $k$ is $F$-finite.
While passing to a perfect closure of $k$ fixes the $F$-finiteness issue,
this operation can change the singularities of $X$ drastically.
To preserve singularities, we prove the following scheme-theoretic version of
the gamma construction of Hochster and Huneke \cite{HH94}.
\begin{alphthm}\label{thm:gammaconstintro}
  Let $X$ be a scheme essentially of finite type over a field
  $k$ of characteristic $p > 0$, and let $\cQ$ be a set
  of properties in the following list:
  local complete intersection, Gorenstein, Cohen--Macaulay, $(S_n)$, regular,
  $(R_n)$, normal, weakly normal, reduced,
  strongly $F$-regular, $F$-pure, $F$-rational, $F$-injective.
  Then, there exists a purely inseparable field extension $k \subseteq k^\Gamma$
  such that $k^\Gamma$ is $F$-finite and such that the projection morphism
  \[
    \pi^\Gamma\colon X \times_k k^\Gamma \longrightarrow X
  \]
  is a homeomorphism that identifies $\cP$ loci for every $\cP \in \cQ$.
\end{alphthm}
\par See \S\ref{sect:fsings} for definitions of $F$-singularities in
the non-$F$-finite setting.
We in fact prove a slightly stronger version of Theorem
\ref{thm:gammaconstintro} that allows $k$ to be replaced by a complete local
ring and allows finitely many schemes instead of just one; see Theorem
\ref{thm:gammaconst}.
We use this added flexibility to prove that
klt and log canonical pairs can be preserved under the gamma construction for
surfaces and threefolds (Corollary \ref{cor:gammaconst}),
providing alternative proofs for the reduction steps in \cite[Thm.\ 3.8]{Tan18}
and \cite[Thm.\ 4.12]{Tan}.
\par We note that parts of Theorems \ref{thm:gammaconstintro} and
\ref{thm:gammaconst} are new even if $X$ is affine.
Namely, the statements for weak normality are completely new, and the statements
for $F$-purity and $F$-injectivity in Theorem \ref{thm:gammaconst} were
previously only known when the scheme $X$ is the spectrum of a complete local
ring \citeleft\citen{EH08}\citemid Lem.\ 2.9\citepunct\citen{Ma14}\citemid
Prop.\ 5.6\citeright.
\medskip
\par In the remainder of this paper, we give applications of the gamma
construction (Theorem \ref{thm:gammaconstintro}) to the theory of asymptotic
invariants of line bundles over arbitrary fields, in the spirit of recent work
of Cutkosky \cite{Cut15}, Fulger--Koll\'ar--Lehmann \cite{FKL16}, Birkar
\cite{Bir17}, and Burgos Gil--Gubler--Jell--K\"unnemann--Martin \cite{BGGJKM}.
See \cite{ELMNP05} for a survey of the theory for smooth complex varieties.
While the main difficulty lies in positive characteristic, we will also prove
statements over fields of characteristic zero that are not necessarily
algebraically closed.
\medskip
\par Our first application provides a characterization of ampleness based on the
asymptotic growth of higher cohomology groups.
It is well known that if $X$ is a projective variety of dimension $n > 0$, then
$h^i(X,\cO_X(mL)) = O(m^n)$ for every Cartier divisor $L$; see \cite[Ex.\
1.2.20]{Laz04a}.
It is natural to ask when cohomology groups have submaximal growth.
The following result says that ample Cartier divisors $L$ are characterized by
having submaximal growth of higher cohomology groups for small perturbations of
$L$.
\begin{alphthm}\label{thm:dfkl41}
  Let $X$ be a projective variety of dimension $n > 0$ over a field $k$.
  Let $L$ be an $\RR$-Cartier divisor on $X$, and consider the following
  property:
  \begin{enumerate}[label=$(\star)$,ref=\star]
    \item\label{thm:dfkl41cond}
      There exists a very ample Cartier divisor $A$ on $X$ and a
      real number $\varepsilon > 0$ such that
      \[
        \widehat{h}^i(X,L-tA) \coloneqq \limsup_{m \to \infty}
        \frac{h^i\bigl(X,\cO_X\bigl(\lceil m(L-tA) \rceil\bigr)\bigr)}{m^n/n!} =
        0
      \]
      for all $i > 0$ and for all $t \in [0,\varepsilon)$.
  \end{enumerate}
  Then, $L$ is ample if and only if $L$ satisfies $(\ref{thm:dfkl41cond})$ for
  some pair $(A,\varepsilon)$.
\end{alphthm}
\par Here, the functions $\widehat{h}^i(X,-)$ are the \textsl{asymptotic higher
cohomological functions} introduced by K\"uronya \cite{Kur06}.
Theorem \ref{thm:dfkl41} was first proved by de~Fernex, K\"uronya, and
Lazarsfeld over the complex numbers \cite[Thm.\ 4.1]{dFKL07}.
In positive characteristic, an interesting aspect of our proof is that it
requires the gamma construction (Theorem \ref{thm:gammaconstintro}) to reduce to
the case when $k$ is $F$-finite.
The main outline of the proof follows that in \cite{dFKL07}, although overcoming
the three problems described above requires care.
\par We note that our motivation for Theorem \ref{thm:dfkl41} comes from
studying Seshadri constants, where Theorem \ref{thm:dfkl41} can be used to show
that Seshadri constants and moving Seshadri constants of $\QQ$-Cartier divisors
can be described via jet separation on projective varieties over arbitrary
fields, without smoothness assumptions \cite[Prop.\ 7.2.10]{Mur19}.
As a result, one can extend \cite[Thm.\ A]{Mur18} to varieties with
either singularities of dense $F$-injective type in characteristic zero, or
varieties with $F$-injective singularities in characteristic $p > 0$, without
any assumptions on the ground field \cite[Thm.\ 7.3.1]{Mur19}.
In \cite[Thm.\ D]{Mur19}, we use a similar argument to prove a new, local
version of the Angehrn--Siu theorem \cite[Thm.\ 0.1]{AS95} in characteristic
zero without the use of Kodaira-type vanishing theorems.
\medskip
\par Our second application is a special case of a conjecture of Boucksom,
Broustet, and Pacienza \cite{BBP13}.
If $D$ is a pseudoeffective $\RR$-Cartier divisor on a projective variety $X$,
then both the \textsl{restricted base locus} $\Bminus(D)$ of $D$, which is a
lower approximation of the stable base locus $\SB(D)$ of $D$, and the
\textsl{non-nef locus} $\NNef(D)$ of $D$, which is defined in terms of
divisorial valuations, are empty if and only if $D$ is nef.
See \S\ref{sect:bminus} for definitions of both invariants.
Boucksom, Broustet, and Pacienza conjectured that these two invariants of $D$
are equal for all pseudoeffective $\RR$-Cartier divisors on normal projective
varieties \cite[Conj.\ 2.7]{BBP13}.
We extend the known cases of their conjecture to projective varieties over
arbitrary fields.
\begin{alphthm}\label{thm:bbpconj}
  Let $X$ be a normal projective variety over a field $k$, and let $D$ be a
  pseudoeffective $\RR$-Cartier divisor on $X$.
  If $\Char k = 0$ and the non-klt locus of $X$ is at most zero-dimensional, or
  if $\Char k = p > 0$ and the non-strongly $F$-regular locus of $X$ is at most
  zero-dimensional, then
  \[
    \Bminus(D) = \NNef(D).
  \]
\end{alphthm}
\par This extends theorems of Nakayama \cite[Lem.\ V.1.9(1)]{Nak04} (in the
smooth case) and Cacciola--Di Biagio \cite[Cor.\ 4.9]{CDB13} over the complex
numbers, and of Musta\c{t}\u{a} \cite[Thm.\ 7.2]{Mus13} (in the regular case)
and Sato \cite[Cor.\ 4.8]{Sat18} over $F$-finite fields of characteristic $p >
0$.
\subsection*{Outline}
This paper is structured as follows:
In \S\ref{sect:defs}, we review some basic material, including the necessary
background on $F$-finiteness, $F$-singularities, and test ideals.
In \S\ref{sect:gammaconst}, we prove Theorem \ref{thm:gammaconst}, which is a
stronger version of Theorem \ref{thm:gammaconstintro}.
We use this stronger version in some applications to the minimal model program
over imperfect fields in \S\ref{sect:tanakaapps}.
The last two sections are devoted to our applications of the gamma construction.
In \S\ref{sect:dfkl}, we prove Theorem \ref{thm:dfkl41} after reviewing some
background on asymptotic cohomological functions.
An important ingredient is a lemma on base loci (Proposition \ref{prop:dfkl31})
analogous to \cite[Prop.\ 3.1]{dFKL07}.
In \S\ref{sect:nakayamabminus}, we prove Theorem \ref{thm:bbpconj} after giving
some background on restricted base loci and non-nef loci.
Finally, in Appendix \ref{app:finj}, we prove some results on $F$-injective
rings for which we could not find a suitable reference, and in Appendix
\ref{app:nonffinfreg}, we describe different notions of strong $F$-regularity
for non-$F$-finite rings.
\subsection*{Notation}
All rings will be commutative with identity.
If $R$ is a ring, then $R^\circ$ denotes the complement of the union of the
minimal primes of $R$.
A \textsl{variety} is a reduced and irreducible scheme that is separated and of
finite type over a field.
A \textsl{complete scheme} is a scheme that is proper over a field.
Intersection products are defined using Euler characteristics;
see \cite[App.\ B]{Kle05}.
\par Let $\kk \in \{\QQ,\RR\}$.
A \textsl{$\kk$-Cartier divisor} (resp.\
\textsl{$\kk$-Weil divisor}) is an element of $\Div_\kk(X) \coloneqq \Div(X)
\otimes_\ZZ \kk$ (resp.\ $\WDiv_\kk(X) \coloneqq \WDiv(X)
\otimes_\ZZ \kk$).
We denote $\kk$-linear equivalence (resp.\ $\kk$-numerical
equivalence) by $\sim_\kk$ (resp.\ $\equiv_\kk$).
We then set $N^1_\kk(X) \coloneqq \Div_\kk(X)/\mathord{\equiv_\kk}$, which is a
finite-dimensional $\kk$-vector space if $X$ is a complete scheme
\cite[Prop.\ 2.3]{Cut15}.
We fix compatible norms $\lVert \cdot \rVert$ on $N^1_\kk(X)$ for $\kk \in
\{\QQ,\RR\}$.
\par If $X$ is a scheme of prime characteristic $p > 0$, then we denote by
$F\colon X \to X$ the \textsl{(absolute) Frobenius morphism,} which is given by
the identity map on points, and the $p$-power map
\[
  \begin{tikzcd}[row sep=0,column sep=1.6em]
    \cO_X(U) \rar & F_*\cO_X(U)\\
    f \rar[mapsto] & f^p
  \end{tikzcd}
\]
on structure sheaves, where $U \subseteq X$ is an open subset.
If $R$ is a ring of prime characteristic $p > 0$, we denote the corresponding ring
homomorphism by $F\colon R \to F_*R$.
For every integer $e \ge 0$, the $e$th iterate of the Frobenius morphisms for
schemes or rings is denoted by $F^e$.
\subsection*{Acknowledgments}
I am grateful to my advisor Mircea Musta\c{t}\u{a} for his constant support and
for several illuminating conversations, and to Mitsuyasu Hashimoto, Melvin
Hochster, Emanuel Reinecke, Karen E. Smith, Daniel Smolkin, Matthew Stevenson,
Shunsuke Takagi, and Farrah Yhee for helpful discussions.
I am particularly grateful to Harold Blum for pointing out applications of
\cite{dFKL07}, to Rankeya Datta for clarifications on strong $F$-regularity for
non-$F$-finite rings, to Alex K\"uronya for talking to me about his results
in \cite{Kur06} and \cite{dFKL07}, and to Linquan Ma, Thomas Polstra, Karl
Schwede, and Kevin Tucker for sharing a preliminary draft of \cite{MPST19} with
me.
I first proved Theorem \ref{thm:dfkl41} in order to prove some results
in joint work with Mihai Fulger, and I would especially like to thank him for
allowing me to write this standalone paper.
Finally, I am indebted to the anonymous referee for useful suggestions that
improved the quality of this paper.

\section{Definitions and preliminaries}\label{sect:defs}
\subsection{Morphisms essentially of finite type}
Recall that a ring homomorphism $A \to B$ is \textsl{essentially of finite type}
if $B$ is isomorphic (as an $A$-algebra) to a localization of an $A$-algebra of
finite type.
The corresponding scheme-theoretic notion is the following:
\begin{citeddef}[{\cite[Def.\ 2.1$(a)$]{Nay09}}]
  Let $f\colon X \to Y$ be a morphism of schemes.
  We say that $f$ is \textsl{locally essentially of finite type} if there is an
  affine open covering $Y = \bigcup_i \Spec A_i$
  such that for every $i$, there is an affine open covering
  \[
    f^{-1}(\Spec A_i) = \bigcup_j \Spec B_{ij}
  \]
  for which the corresponding ring homomorphisms $A_i \to B_{ij}$ are
  essentially of finite type.
  We say that $f$ is \textsl{essentially of finite type} if it is
  locally essentially of finite type and quasi-compact.
\end{citeddef}
\par The class of morphisms (locally) essentially of finite type is closed
under composition and base change \cite[(2.2)]{Nay09}.
\subsection{Base loci}
In this subsection, we define the base ideal of a Cartier divisor and related
objects.
\begin{definition}[see {\cite[Def.\ 1.1.8]{Laz04a}}]\label{def:baseideal}
  Let $X$ be a complete scheme over a field $k$, and let $D$ be a Cartier
  divisor.
  The \textsl{complete linear series} associated to $D$ is the projective space
  $\lvert D \rvert \coloneqq \PP(H^0(X,\cO_X(D))^\vee)$
  of one-dimensional subspaces of $H^0(X,\cO_X(D))$.
  The \textsl{base ideal} of $D$ is
  \begin{equation}\label{eq:baseidealeval}
    \fb\bigl(\lvert D \rvert\bigr) \coloneqq \im\bigl(H^0\bigl(X,\cO_X(D)\bigr)
    \otimes_k \cO_X(-D) \xrightarrow{\eval} \cO_X \bigr).
  \end{equation}
  The \textsl{base scheme} $\Bs(\lvert D \rvert)$ of $D$ is
  the closed subscheme of $X$ defined by $\fb(\lvert D \rvert)$, and the
  \textsl{base locus} of $D$ is the underlying closed subset
  $\Bs(\lvert D \rvert)_\red$.
\end{definition}
\par We will need the following description for how base ideals transform under
birational morphisms.
\begin{lemma}\label{lem:baselocusnormal}
  Let $f\colon X' \to X$ be a birational morphism between complete varieties,
  where $X$ is normal.
  Then, for every Cartier divisor $D$ on $X$, we have
  $f^{-1}\fb(\lvert D \rvert)\cdot\cO_{X'} = \fb(\lvert f^*D \rvert)$.
\end{lemma}
\begin{proof}
  Since $X$ is normal, we have $f_*\cO_{X'} \simeq \cO_X$ \cite[Proof of Cor.\
  III.11.4]{Har77}.
  By the projection formula, we then have $H^0(X,\cO_X(D)) \simeq
  H^0(X',\cO_{X'}(f^*D))$, and the lemma then follows by pulling back the
  evaluation map \eqref{eq:baseidealeval}.
\end{proof}
Next, we define a stable version of the base locus.
\begin{definition}[see {\cite[Def.\ 2.1.20]{Laz04a}}]\label{def:stablebaselocus}
  Let $X$ be a complete scheme over a field, and let $D$ be a Cartier
  divisor on $X$.
  The \textsl{stable base locus} of $D$ is the closed subset
  \begin{equation}\label{eq:defstablebaselocus}
    \SB(D) \coloneqq \bigcap_m \Bs\bigl( \lvert mD \rvert \bigr)_\red
  \end{equation}
  of $X$, where the intersection runs over every integer $m > 0$.
  The noetherian property implies $\SB(D) = \SB(nD)$ for every
  integer $n > 0$ \cite[Ex.\ 2.1.23]{Laz04a}, hence the formula
  \eqref{eq:defstablebaselocus} can be used for $\QQ$-Cartier divisors
  $D$ by taking the intersection over every integer $m > 0$ such that
  $mD$ is integral.
\end{definition}
The stable base locus is not a numerical invariant of $D$ \cite[Ex.\
10.3.3]{Laz04b}.
In \S\ref{sect:bminus}, we will define the \textsl{restricted base locus}
$\Bminus(D)$, which is a numerically invariant approximation of $\SB(D)$.
\subsection{\textit{F}-finite schemes}
As mentioned in \S\ref{sect:intro}, in positive characteristic, one often needs
to restrict or reduce to the case when the Frobenius morphism is finite.
We isolate this class of schemes.
\begin{definition}
  Let $X$ be a scheme of prime characteristic $p > 0$.
  We say that $X$ is \textsl{$F$-finite} if the (absolute) Frobenius morphism
  $F\colon X \to X$ is finite.
  We say that a ring $R$ of prime characteristic $p > 0$ is \textsl{$F$-finite} if
  $\Spec R$ is $F$-finite, or equivalently if $F\colon R \to F_*R$ is
  module-finite.
\end{definition}
\par Note that a field $k$ is $F$-finite if and only if $[k:k^p] < \infty$.
$F$-finite schemes are ubiquitous in geometric contexts because of the
following:
\begin{example}[see {\cite[p.\ 999]{Kun76}}]\label{ex:eftoverffin}
  If $X$ is a scheme that is locally essentially of finite type over an
  $F$-finite scheme of prime characteristic $p > 0$, then $X$ is $F$-finite.
  In particular, schemes essentially of finite type over perfect or $F$-finite
  fields are $F$-finite.
\end{example}
\par If a scheme $X$ of prime characteristic $p > 0$ is $F$-finite, then Grothendieck
duality can be applied to the Frobenius morphism since it is finite
\cite[III.6]{Har66}.
The $F$-finiteness condition implies other desirable conditions as well.
\begin{citedthm}[{\citeleft\citen{Kun76}\citemid Thm.\
  2.5\citepunct\citen{Gab04}\citemid Rem.\ 13.6\citeright}]
  \label{thm:ffiniteaffine}
  Let $R$ be a noetherian $F$-finite ring of prime characteristic $p > 0$.
  Then, $R$ is excellent and is isomorphic to a quotient of a regular
  ring of finite Krull dimension.
  In particular, $R$ admits a dualizing complex $\omega_R^\bullet$.
\end{citedthm}
\subsection{\textit{F}-singularities}\label{sect:fsings}
We review some classes of singularities defined using the Frobenius morphism.
See \cite{TW18} for a survey, and see Appendix \ref{app:nonffinfreg} for more
material on strong $F$-regularity for non-$F$-finite rings.
Recall that a ring homomorphism $R \to S$ is \textsl{pure}
if the homomorphism $R \otimes_R M \to S \otimes_R M$ is injective for every
$R$-module $M$.
\begin{citeddef}[{\citeleft\citen{Has10}\citemid Def.\ 
  3.3\citepunct\citen{HR76}\citemid p.\ 121\citeright}]\label{def:freg}
  Let $R$ be a noetherian ring of prime characteristic $p > 0$.
  For every $c \in R$ and every integer $e > 0$, we denote by $\lambda^e_c$ the
  composition
  \[
    R \overset{F^e}{\longrightarrow} F_*^eR \xrightarrow{F_*^e(-\cdot c)}
    F_*^eR.
  \]
  If $c \in R$, then following \cite[Def.\ 6.1.1]{DS16} we say that $R$ is
  \textsl{$F$-pure along $c$} if $\lambda^e_c$ is pure for some $e > 0$, and
  that
  \begin{enumerate}[label=$(\alph*)$,ref=\alph*]
    \item $R$ is \textsl{strongly $F$-regular} if every localization $R_\fp$ of
      $R$ is $F$-pure along every $c \in R_\fp^\circ$;
      and\label{def:fsingspurereg}
    \item $R$ is \textsl{$F$-pure} if $R$ is $F$-pure along $1 \in R$.
  \end{enumerate}
  Note that $(\ref{def:fsingspurereg})$ is not the usual definition (Definition
  \ref{def:fregapp}$(\ref{def:fsingssplitreg})$) for strong $F$-regularity,
  which coincides with ours for $F$-finite rings.
  See Appendix \ref{app:nonffinfreg} for a description of the relationship
  between different notions of strong $F$-regularity for non-$F$-finite rings.
\end{citeddef}
\par To define $F$-rationality, we recall that if $R$
is a noetherian ring, then a sequence of elements $x_1,x_2,\ldots,x_n \in R$ is
a \textsl{sequence of parameters} if for every prime ideal $\fp$
containing $(x_1,x_2,\ldots,x_n)$, the images of $x_1,x_2,\ldots,x_n$ in $R_\fp$
are part of a system of parameters in $R_\fp$ \cite[Def.\ 2.1]{HH90}.
\begin{citeddef}[{\cite[Def.\ 1.10]{FW89}}]
  A noetherian ring of prime characteristic $p > 0$ is \textsl{$F$-rational} if every
  ideal generated by a sequence of parameters in $R$ is tightly closed in
  $R$.
\end{citeddef}
\par See \cite[Def.\ 3.1]{HH90} for the definition of tight closure.
Finally, we define $F$-injective singularities.
\begin{citeddef}[{\cite[Def.\ on p.\ 473]{Fed83}}]\label{def:finj}
  A noetherian ring $R$ of prime characteristic $p > 0$ is \textsl{$F$-injective} if,
  for every maximal ideal $\fm \subseteq R$, the $R_\fm$-module homomorphism
  $H^i_\fm(F) \colon H^i_\fm(R_\fm) \to H^i_\fm(F_*R_\fm)$ induced by Frobenius
  is injective for all $i$.
\end{citeddef}
\par We will prove some basic results about $F$-injective rings in
Appendix \ref{app:finj}.
\medskip
\par The relationship between these classes of singularities can be summarized
as follows:
\[
  \begin{tikzcd}[column sep=7.5em]
    \text{regular} \rar[Rightarrow]{\text{\cite[Thm.\ 6.2.1]{DS16}}}
    & \text{strongly $F$-regular} \rar[Rightarrow]{\text{\cite[Cor.\
    3.7]{Has10}}} \dar[Rightarrow,swap]{\text{Def.}} & \text{$F$-rational}
    \dar[Rightarrow]{\text{\cite[Prop.\ A.3$(iii)$]{DM}}}\\
    & \text{$F$-pure} \rar[Rightarrow]{\text{\cite[Lem.\ 3.3]{Fed83}}}
    & \text{$F$-injective}
  \end{tikzcd}
\]
\subsection{Test ideals}
We review the theory of test ideals, which are the positive characteristic
analogues of multiplier ideals.
We recall that following \cite[Def.\ 2.4.14]{Laz04a}, a collection
$\fa_\bullet \coloneqq \{\fa_m\}_{m=1}^\infty$ of coherent ideal sheaves $\fa_m
\subseteq \cO_X$ on a locally noetherian scheme $X$ is a \textsl{graded family
of ideals} if $\fa_m \cdot \fa_n \subseteq \fa_{m+n}$ for all $m,n \ge 1$.
We now fix the following notational conventions for pairs.
\begin{definition}[cf.\ {\cite[Def.\ 2.3]{Sch10}}]
  A \textsl{pair} $(X,\fa_\bullet^\lambda)$ consists of
  \begin{enumerate}[label=$(\roman*)$]
    \item an excellent reduced noetherian scheme $X$; and
    \item a symbol $\fa_\bullet^\lambda$ where $\fa_\bullet$ is a graded family
      of ideals on $X$ such that for every open affine subset $U = \Spec R
      \subseteq X$, we have $\fa_m(U) \cap R^\circ \ne \emptyset$ for some $m >
      0$, and $\lambda$ is a positive real number.
  \end{enumerate}
  We drop $\lambda$ from our notation if $\lambda = 1$.
  If $\fa_\bullet = \{\fa^m\}_{m=1}^\infty$ for some fixed ideal sheaf
  $\fa$, then we denote the pair by $(X,\fa^t)$.
  If $X = \Spec R$ for a ring $R$, then we denote
  the pair by $(R,\fa_\bullet^\lambda)$.
\end{definition}
\par We now define test ideals for $F$-finite schemes of prime
characteristic $p > 0$.
See \cite{ST12} and \cite[\S5]{TW18} for overviews of the theory.
We take Schwede's characterization of test ideals via $F$-compatibility
\cite{Sch10} as our definition.
\begin{citeddef}[{\cite[Def.\ 3.1 and Thm.\ 6.3]{Sch10}}]
  Let $(R,\fa^t)$ be a pair such that $R$ is an $F$-finite ring of prime
  characteristic $p > 0$.
  An ideal $J \subseteq R$ is \textsl{uniformly
  $(\fa^t,F)$-compatible} if for every integer $e > 0$ and
  every $\varphi \in \Hom_R(F^e_*R,R)$, we have
  \[
    \varphi\Bigl(F^e_*\bigl(J \cdot \fa^{\lceil t(p^e-1)\rceil}
    \bigr)\Bigr) \subseteq J.
  \]
  \par Now let $(X,\fa^t)$ be a pair such that $X$ is an $F$-finite scheme of
  prime characteristic $p > 0$.
  The \textsl{test ideal} $\tau(X,\fa^t)$ is defined locally on each affine
  open subset $U = \Spec R\subseteq X$ as the smallest ideal that is
  uniformly $(\fa^t,F)$-compatible and intersects $R^\circ$.
  \par We often drop $X$ from our notation if it is clear from context.
  Similarly, we drop $\fa^t$ or $\fa_\bullet^\lambda$ from our
  notation when working with the scheme itself.
\end{citeddef}
\par The test ideal is well defined by \cite[Prop.\ 3.23$(ii)$]{Sch11}, and exists
by \cite[Thm.\ 3.18]{Sch11}.
Formal properties analogous to those for multiplier ideals hold for test ideals;
see \cite[Prop.\ 5.6]{TW18}.
We can therefore define the following asymptotic version of test ideals:
\begin{citeddef}[{\cite[Prop.-Def.\ 2.16]{Sat18}}]
  Let $(X,\fa_\bullet^\lambda)$ be a pair such that $X$ is $F$-finite
  of prime characteristic $p > 0$.
  The \textsl{asymptotic test ideal} $\tau(X,\fa_\bullet^{\lambda})$ is
  defined to be $\tau(X,\fa_{m}^{\lambda/m})$ for sufficiently large and
  divisible $m$.
\end{citeddef}
\par The subadditivity theorem holds for both test ideals and asymptotic test ideals
on regular complete local rings \cite[Thm.\ 6.10(2)]{HY03}, and therefore also
holds on all regular $F$-finite schemes $X$, since for $F$-finite schemes, the
formation of test ideals is compatible with localization and completion
\cite[Props.\ 3.1 and 3.2]{HT04}.
\medskip
\par The following example will be the most important in our applications.
\begin{example}[see {\cite[Def.\ 2.36]{Sat18}}]
  Let $X$ be a complete reduced scheme over an $F$-finite field of
  characteristic $p >0$.
  If $D$ is a Cartier divisor such that $H^0(X,\cO_X(mD)) \ne 0$ for some
  positive integer $m$, then for every real number $t > 0$, we set
  \begin{align*}
    \tau\bigl(X,t \cdot \lvert D \rvert\bigr) &\coloneqq
    \tau\bigl(X,\fb\bigl(\lvert D \rvert\bigr)^t\bigr).
    \intertext{If $D$ is a $\QQ$-Cartier divisor such that $H^0(X,\cO_X(mD)) \ne
    0$ for some sufficiently divisible $m > 0$, then for every real number
    $\lambda > 0$, we set}
    \tau\bigl(X,\lambda \cdot \lVert D \rVert\bigr) &\coloneqq
    \tau\bigl(X,\fa_\bullet(D)^\lambda\bigr),
  \end{align*}
  where $\fa_m(D) = \fb(\lvert mD \rvert)$ if $mD$ is integral, and $0$
  otherwise.
  See \cite[Ex.\ 2.4.16$(ii)$]{Laz04a}.
\end{example}

\section{The gamma construction of Hochster--Huneke}\label{sect:gammaconst}
Our goal in this section is to prove Theorem \ref{thm:gammaconstintro}, which is
a scheme-theoretic version of the gamma construction of Hochster and Huneke
\cite{HH94}.
Hochster and Huneke first introduced the gamma construction in order to prove
that test elements (in the sense of tight closure) exist for rings that are
essentially of finite type over an excellent local ring of prime characteristic $p >
0$.
To the best of our knowledge, however, their construction has not been
applied explicitly in a geometric context.
\par As mentioned in \S\ref{sect:intro}, Theorem \ref{thm:gammaconstintro}
provides a systematic way to reduce to the case when the ground field $k$ is
$F$-finite.
We will in fact show a more general result (Theorem \ref{thm:gammaconst}), which
allows for the ground field $k$ to be replaced by a complete local ring, and
allows finitely many schemes instead of just one.
After proving Theorems \ref{thm:gammaconstintro} and \ref{thm:gammaconst} in
\S\ref{sect:gammaconstproof}, we prove that the $F$-pure locus is open in
schemes essentially of finite type over excellent local rings (Corollary
\ref{cor:fpurelocusopen}).
We then give some basic applications of Theorem \ref{thm:gammaconst} to the
minimal model program over imperfect fields in \S\ref{sect:tanakaapps}.
\subsection{The construction and proof of Theorem \ref{thm:gammaconstintro}}
\label{sect:gammaconstproof}
We start with the following account of Hochster and Huneke's construction.
\begin{citedconstr}[{\cite[(6.7) and (6.11)]{HH94}}]\label{constr:gamma}
  Let $(A,\fm,k)$ be a noetherian complete local ring of prime characteristic $p > 0$.
  By the Cohen structure theorem, we
  may identify $k$ with a coefficient field $k \subseteq A$.
  Moreover, by Zorn's lemma (see \cite[p.\ 202]{Mat89}), we
  may choose a \textsl{$p$-basis} $\Lambda$ for $k$, which is a subset $\Lambda
  \subseteq k$ such that $k = k^p(\Lambda)$, and such that for every finite
  subset $\Sigma \subseteq \Lambda$ with $s$ elements, we have $[k^p(\Sigma) :
  k^p] = p^s$.
  \par Now let $\Gamma \subseteq \Lambda$ be a cofinite subset, i.e., a subset
  $\Gamma$ of $\Lambda$ such that $\Lambda \smallsetminus \Gamma$ is a finite
  set.
  For each integer $e \ge 0$, consider the subfield
  \[
    k_e^\Gamma = k\bigl[\lambda^{1/p^e}\bigr]_{\lambda \in \Gamma}
    \subseteq k_{\perf}
  \]
  of a perfect closure $k_{\perf}$ of $k$.
  These form an ascending chain, and we then set
  \[
    A^\Gamma \coloneqq \varinjlim_e\Bigl(k_e^\Gamma\llbracket A \rrbracket\Bigr),
  \]
  where $k_e^\Gamma\llbracket A \rrbracket$ is the completion of $k_e^\Gamma
  \otimes_k A$ at the extended ideal $\mathfrak{m}\cdot(k_e^\Gamma \otimes_k
  A)$.
  Note that if $A = k$ is a field, then $A^\Gamma = k^\Gamma$ is a field
  by construction.
  \par Finally, let $X$ be a scheme essentially of finite type over $A$, and
  consider two cofinite subsets $\Gamma \subseteq \Lambda$ and
  $\Gamma' \subseteq \Lambda$ such that $\Gamma \subseteq \Gamma'$.
  We then have the following commutative diagram whose vertical faces are
  cartesian:
  \[
    \begin{tikzcd}[column sep=-0.6em]
      X^{\Gamma'} \arrow{dr}{\pi^{\Gamma'}}\arrow{rr}{\pi^{\Gamma\Gamma'}}
      \arrow{dd} & & X^\Gamma \arrow{dl}[swap]{\pi^\Gamma}\arrow{dd}\\
      & X\\[-1.35em]
      \Spec A^{\Gamma'}\arrow{rr}\arrow{dr} & & \Spec A^\Gamma \arrow{dl}\\
      & \Spec A\arrow[leftarrow,crossing over]{uu}
    \end{tikzcd}
  \]
\end{citedconstr}
\par We list some elementary properties of the gamma construction.
\begin{lemma}
  \label{lem:gammaconstbasic}
  Fix notation as in Construction \ref{constr:gamma}, and let $\Gamma \subseteq
  \Lambda$ be a cofinite subset.
  \begin{enumerate}[label=$(\roman*)$,ref=\roman*]
    \item The ring $A^\Gamma$ and the scheme $X^\Gamma$ are noetherian and
      $F$-finite.
      \label{lem:gammaconstbasicffin}
    \item The morphism $\pi^\Gamma$ is a faithfully flat universal homeomorphism
      with local complete intersection fibers.\label{lem:gammaconstbasicfflat}
    \item Given a cofinite subset $\Gamma \subseteq \Gamma'$, the morphism
      $\pi^{\Gamma\Gamma'}$ is a faithfully flat universal
      homeomorphism.\label{lem:gammaconstbasictransitionfflat}
  \end{enumerate}
\end{lemma}
\begin{proof}
  The ring $A^\Gamma$ is noetherian and $F$-finite \cite[(6.11)]{HH94}, hence
  $X^\Gamma$ is also by Example \ref{ex:eftoverffin} and the fact that morphisms
  essentially of finite type are preserved under base change
  \cite[(2.2)]{Nay09}.
  The ring extensions $A \subseteq A^\Gamma$ and $A^\Gamma \subseteq
  A^{\Gamma'}$ are purely inseparable and faithfully flat \cite[(6.11)]{HH94},
  hence induce faithfully flat universal homeomorphisms on spectra \cite[Prop.\
  2.4.5$(i)$]{EGAIV2}.
  Thus, the morphisms $\pi^\Gamma$ and $\pi^{\Gamma\Gamma'}$ are faithfully flat
  universal homeomorphisms by base change.
  Finally, the ring extension $A \subseteq A^\Gamma$ is flat with local complete
  intersection fibers \cite[Lem.\ 3.19]{Has10}, hence $\pi^\Gamma$ is also by
  base change \cite[Cor.\ 4]{Avr75}.
\end{proof}
\par Our goal now is to prove that if a local property of schemes satisfies certain
conditions, then the property is preserved when passing from $X$ to $X^\Gamma$
for ``small enough'' $\Gamma$.
For a scheme $X$ and a property $\cP$ of local rings on $X$, the
\textsl{$\cP$ locus} of $X$ is $\cP(X) \coloneqq \{x \in X \mid
\cO_{X,x}\ \text{is $\cP$}\}$.
\begin{proposition}\label{prop:gammaconstconds}
  Fix notation as in Construction \ref{constr:gamma}, and let $\cP$ be a
  property of local rings of prime characteristic $p > 0$.
  \begin{enumerate}[label=$(\roman*)$,ref=\roman*]
    \item Suppose that for every flat local homomorphism $B \to C$ of noetherian
      local rings with local complete intersection fibers, if $B$ is $\cP$, then
      $C$ is $\cP$.
      Then, $\pi^\Gamma(\cP(X^\Gamma)) = \cP(X)$ for every cofinite subset
      $\Gamma \subseteq \Lambda$.
      \label{prop:gammaconstcm}
    \item Consider the following conditions:
      \begin{enumerate}[label=$(\Gamma\arabic*)$,ref=\ensuremath{\Gamma}\arabic*]
        \item If $B$ is a noetherian $F$-finite ring of prime characteristic $p > 0$,
          then $\cP(\Spec B)$ is open.\label{axiom:gammaopen}
        \item For every flat local homomorphism $B \to C$ of noetherian
          local rings of prime characteristic $p > 0$ with zero-dimensional
          fibers, if $C$ is $\cP$, then $B$ is
          $\cP$.\label{axiom:gammadescent}
        \item For every local ring $B$ essentially of finite type over $A$, if
          $B$ is $\cP$, then there exists a cofinite subset $\Gamma_1 \subseteq
          \Lambda$ such that $B^\Gamma$ is $\cP$ for every cofinite subset
          $\Gamma \subseteq \Gamma_1$.\label{axiom:gammaascent}
        \item[$(\ref*{axiom:gammaascent}')$]
          \refstepcounter{enumi}
          \makeatletter
          \def\@currentlabel{\ref*{axiom:gammaascent}\ensuremath{'}}
          \makeatother
          For every flat local homomorphism $B \to C$ of noetherian local rings
          of prime characteristic $p > 0$ such that the closed fiber is a field, if
          $B$ is $\cP$, then $C$ is $\cP$.\label{axiom:gammaascentprime}
      \end{enumerate}
      If $\cP$ satisfies $(\ref{axiom:gammaopen})$,
      $(\ref{axiom:gammadescent})$, and one of either
      $(\ref{axiom:gammaascent})$ or
      $(\ref{axiom:gammaascentprime})$, then there exists a cofinite subset
      $\Gamma_0 \subseteq \Lambda$ such that $\pi^\Gamma(\cP(X^\Gamma)) =
      \cP(X)$ for every cofinite subset $\Gamma \subseteq \Gamma_0$.
      \label{prop:gammaconstaxiomatic}
  \end{enumerate}
\end{proposition}
\begin{proof}
  For $(\ref{prop:gammaconstcm})$, it suffices to note that $\pi^\Gamma$ is
  faithfully flat with local complete intersection fibers by Lemma
  \ref{lem:gammaconstbasic}$(\ref{lem:gammaconstbasicfflat})$.
  \par For $(\ref{prop:gammaconstaxiomatic})$, we first note that
  $(\ref{axiom:gammaascentprime})$ implies $(\ref{axiom:gammaascent})$, since
  there exists a cofinite subset $\Gamma_1 \subseteq \Lambda$ such that the
  closed fiber is a field for every cofinite subset $\Gamma \subseteq \Gamma_1$
  by \cite[Lem.\ 6.13$(b)$]{HH94}.
  From now on, we therefore assume that $\cP$ satisfies
  $(\ref{axiom:gammaopen})$, $(\ref{axiom:gammadescent})$, and
  $(\ref{axiom:gammaascent})$.
  \par For every cofinite subset $\Gamma \subseteq \Lambda$, the set
  $\cP(X^\Gamma)$ is open by $(\ref{axiom:gammaopen})$ since $X^\Gamma$ is
  noetherian and $F$-finite by Lemma
  \ref{lem:gammaconstbasic}$(\ref{lem:gammaconstbasicffin})$.
  Moreover, the morphisms $\pi^\Gamma$ and $\pi^{\Gamma\Gamma'}$ are faithfully
  flat universal homeomorphisms for every cofinite subset $\Gamma' \subseteq
  \Lambda$ such that $\Gamma \subseteq \Gamma'$ by Lemmas
  \ref{lem:gammaconstbasic}$(\ref{lem:gammaconstbasicfflat})$ and 
  \ref{lem:gammaconstbasic}$(\ref{lem:gammaconstbasictransitionfflat})$, hence
  by $(\ref{axiom:gammadescent})$, we have the inclusions
  \begin{equation}\label{eq:ugammaincl}
    \cP(X) \supseteq \pi^\Gamma\bigl(\cP(X^\Gamma)\bigr) \supseteq
    \pi^{\Gamma'}\bigl(\cP(X^{\Gamma'})\bigr)
  \end{equation}
  in $X$, where $\pi^\Gamma(\cP(X^\Gamma))$ and
  $\pi^{\Gamma'}(\cP(X^{\Gamma'}))$ are open.
  Since $X$ is noetherian, it satisfies the ascending chain condition on the
  open sets $\pi^\Gamma(\cP(X^\Gamma))$, hence we can choose a cofinite subset
  $\Gamma_0 \subseteq \Lambda$ such that $\pi^{\Gamma_0}(\cP(X^{\Gamma_0}))$ is
  maximal with respect to inclusion.
  \par We claim that $\cP(X) = \pi^{\Gamma_0}(\cP(X^{\Gamma_0}))$ for every
  cofinite subset $\Gamma \subseteq \Gamma_0$.
  By \eqref{eq:ugammaincl}, it suffices to show the inclusion $\subseteq$.
  Suppose there exists $x \in \cP(X) \smallsetminus
  \pi^{\Gamma_0}(\cP(X^{\Gamma_0}))$.
  By $(\ref{axiom:gammaascent})$, there exists a cofinite subset $\Gamma_1
  \subseteq \Lambda$ such that $(\pi^\Gamma)^{-1}(x) \in \cP(X^\Gamma)$ for
  every cofinite subset $\Gamma \subseteq \Gamma_1$.
  Choosing $\Gamma = \Gamma_0 \cap \Gamma_1$,
  we have $x \in \pi^\Gamma(\cP(X^\Gamma)) \smallsetminus
  \pi^{\Gamma_0}(\cP(X^{\Gamma_0}))$, contradicting the maximality of
  $\pi^{\Gamma_0}(\cP(X^{\Gamma_0}))$.
\end{proof}
\par We now prove that the properties in Theorem \ref{thm:gammaconstintro} are
preserved when passing to $X^\Gamma$.
Special cases of the following  result appear in \cite[Lem.\ 6.13]{HH94},
\cite[Thm.\ 2.2]{Vel95}, \cite[Lem.\ 2.9]{EH08}, \cite[Lems.\ 3.23 and
3.30]{Has10}, and \cite[Prop.\ 5.6]{Ma14}.
\begin{theorem}\label{thm:gammaconst}
  Fix notation as in Construction \ref{constr:gamma}.
  \begin{enumerate}[label=$(\roman*)$,ref=\roman*]
    \item For every cofinite subset $\Gamma \subseteq \Lambda$, the map
      $\pi^\Gamma$ identifies local complete intersection, Gorenstein,
      Cohen--Macaulay, and $(S_n)$ loci.\label{thm:gammaconstcm}
    \item There exists a cofinite subset $\Gamma_0 \subseteq
      \Lambda$ such that $\pi^\Gamma$ identifies regular (resp.\ $(R_n)$, normal,
      weakly normal, reduced, strongly $F$-regular, $F$-pure, $F$-rational,
      $F$-injective) loci for every cofinite subset $\Gamma \subseteq
      \Gamma_0$.\label{thm:gammaconstreg}
  \end{enumerate}
\end{theorem}
\par Note that Theorem \ref{thm:gammaconst} implies Theorem \ref{thm:gammaconstintro}
since if $A$ is a
field, then $A^\Gamma$ is also by Construction \ref{constr:gamma}, and moreover
if one wants to preserve more than one property at once, then it suffices to
intersect the various $\Gamma_0$ for the different properties.
\begin{proof}
  For $(\ref{thm:gammaconstcm})$, it suffices to note that these properties
  satisfy the condition in Proposition
  \ref{prop:gammaconstconds}$(\ref{prop:gammaconstcm})$ by \cite[Cor.\
  2]{Avr75} and \cite[Thm.\ 23.4, Cor.\ to Thm.\ 23.3, and Thm.\
  23.9$(iii)$]{Mat89}, respectively.
  \par We now prove $(\ref{thm:gammaconstreg})$.
  We first note that $(\ref{thm:gammaconstreg})$ holds for regularity since
  $(\ref{axiom:gammaopen})$ holds by the excellence of $X^\Gamma$, and
  $(\ref{axiom:gammadescent})$ and $(\ref{axiom:gammaascentprime})$ hold by
  \cite[Thm.\ 23.7]{Mat89}.
  Since $\pi^\Gamma$ preserves the dimension of local rings, we therefore see
  that $(\ref{thm:gammaconstreg})$ holds for $(R_n)$.
  $(\ref{thm:gammaconstreg})$ for normality and reducedness then follows from
  $(\ref{thm:gammaconstcm})$ since they are equivalent to $(R_1)+(S_2)$ and
  $(R_0)+(S_1)$, respectively.
  \par To prove $(\ref{thm:gammaconstreg})$ holds in the remaining cases, we
  check the conditions in Proposition
  \ref{prop:gammaconstconds}$(\ref{prop:gammaconstaxiomatic})$.
  For weak normality, $(\ref{axiom:gammaopen})$ holds by \cite[Thm.\
  7.1.3]{BF93}, and $(\ref{axiom:gammadescent})$ holds by \cite[Cor.\
  II.2]{Man80}.
  To show that $(\ref{axiom:gammaascent})$ holds, recall by \cite[Thm.\
  I.6]{Man80} that a reduced ring $B$ is weakly normal if and only if
  \begin{equation}\label{eq:manaresiwnchar}
    \begin{tikzcd}[column sep=large]
      B \rar &[-2.125em] B^\nu \rar[shift left=3pt]{b \mapsto b \otimes 1}
      \rar[shift right=3pt,swap]{b \mapsto 1 \otimes b} & (B^\nu \otimes_B
      B^\nu)_\red
    \end{tikzcd}
  \end{equation}
  is an equalizer diagram, where $B^\nu$ is the normalization of $B$.
  Now suppose $B$ is weakly normal, and let $\Gamma_1 \subseteq \Lambda$ be a
  cofinite subset such that $B^\Gamma$ is reduced, $(B^\nu)^\Gamma$ is normal,
  and $((B^\nu \otimes_B B^\nu)_\red)^\Gamma$ is reduced for every cofinite
  subset $\Gamma \subseteq \Gamma_1$; such a $\Gamma_1$ exists by the
  previous paragraph.
  We claim that $B^\Gamma$ is weakly normal for every $\Gamma \subseteq
  \Gamma_1$ cofinite in $\Lambda$.
  Since \eqref{eq:manaresiwnchar} is an equalizer diagram and $A \subseteq
  A^\Gamma$ is flat, the diagram
  \[
    \begin{tikzcd}[column sep=large]
      B^\Gamma \rar &[-2.125em] (B^\nu)^\Gamma \rar[shift left=3pt]{b \mapsto b
      \otimes 1} \rar[shift right=3pt,swap]{b \mapsto 1 \otimes b} &
      \bigl((B^\nu \otimes_B B^\nu)_\red\bigr)^\Gamma
    \end{tikzcd}
  \]
  is an equalizer diagram.
  Moreover, since $B^\Gamma \subseteq (B^\nu)^\Gamma$ is an integral
  extension of rings with the same total ring of fractions, and $(B^\nu)^\Gamma$
  is normal, we see that $(B^\nu)^\Gamma = (B^\Gamma)^\nu$.
  Finally, $((B^\nu \otimes_B B^\nu)_\red)^\Gamma$ is reduced, hence we have the
  natural isomorphism
  \[
    \bigl((B^\nu \otimes_B B^\nu)_\red\bigr)^\Gamma \simeq \bigl((B^\Gamma)^\nu
    \otimes_{B^\Gamma} (B^\Gamma)^\nu\bigr)_\red.
  \]
  Thus, since the analogue of \eqref{eq:manaresiwnchar} with $B$ replaced
  by $B^\Gamma$ is an equalizer diagram, we see that $B^\Gamma$ is weakly normal
  for every $\Gamma \subseteq \Gamma_1$ cofinite in $\Lambda$, hence
  $(\ref{axiom:gammaascent})$ holds for weak normality.
  \par We now prove $(\ref{thm:gammaconstreg})$ for strong $F$-regularity,
  $F$-purity, and $F$-rationality.
  First, $(\ref{axiom:gammaopen})$ holds for strong $F$-regularity by
  \cite[Lem.\ 3.29]{Has10}, and the same argument shows that
  $(\ref{axiom:gammaopen})$ holds for $F$-purity since the $F$-pure and
  $F$-split loci coincide for $F$-finite rings \cite[Cor.\ 5.3]{HR76}.
  Next, $(\ref{axiom:gammaopen})$ for $F$-rationality holds by \cite[Thm.\
  1.11]{Vel95} since the reduced locus is open and reduced $F$-finite rings are
  admissible in the sense of \cite[Def.\ 1.5]{Vel95} by Theorem
  \ref{thm:ffiniteaffine}.
  It then suffices to note that $(\ref{axiom:gammadescent})$ holds by
  \cite[Lem.\ 3.17]{Has10}, \cite[Prop.\ 5.13]{HR76}, and \cite[(6) on p.\
  440]{Vel95}, respectively, and $(\ref{axiom:gammaascent})$ holds by
  \cite[Cor.\ 3.31]{Has10}, \cite[Prop.\ 5.4]{Ma14}, and \cite[Lem.\
  2.3]{Vel95}, respectively.
  \par Finally, we prove $(\ref{thm:gammaconstreg})$ for $F$-injectivity.
  First, $(\ref{axiom:gammaopen})$ and $(\ref{axiom:gammadescent})$ hold by
  Lemmas \ref{lem:finjopen} and \ref{lem:fsingsflatextfinj}, respectively.
  The proof of \cite[Lem.\ 2.9$(b)$]{EH08} implies
  $(\ref{axiom:gammaascent})$, since the residue field of $B$ is a finite
  extension of $k$, hence socles of artinian $B$-modules are finite-dimensional
  $k$-vector spaces.
\end{proof}
\par We have the following consequence of Theorem \ref{thm:gammaconst}, which was
first attributed to Hoshi in \cite[Thm.\ 3.2]{Has10proc}.
Note that the analogous statements for strong $F$-regularity and $F$-rationality
appear in \cite[Prop.\ 3.33]{Has10} and \cite[Thm.\ 3.5]{Vel95}, respectively.
\begin{corollary}\label{cor:fpurelocusopen}
  Let $X$ be a scheme essentially of finite type over a local $G$-ring
  $(A,\fm)$ of prime characteristic $p > 0$.
  Then, the $F$-pure locus is open in $X$.
\end{corollary}
\par Recall that a noetherian ring $R$ is a
\textsl{$G$-ring} if, for every prime ideal $\fp \subseteq R$, the completion
homomorphism $R_\fp \to \widehat{R_\fp}$ is regular; see \cite[pp.\
255--256]{Mat89} for the definitions of $G$-rings and of regular homomorphisms.
Excellent rings are $G$-rings by definition; see \cite[Def.\ on p.\ 260]{Mat89}.
\begin{proof}
  Let $A \to \widehat{A}$ be the completion of $A$ at $\fm$, and let $\Lambda$
  be a $p$-basis for $\widehat{A}/\fm\widehat{A}$ as in Construction
  \ref{constr:gamma}.
  For every cofinite subset $\Gamma \subseteq \Lambda$, consider the
  commutative diagram
  \[
    \begin{tikzcd}
      X \times_A \widehat{A}^\Gamma \rar{\pi^\Gamma}\dar & X \times_A
      \widehat{A} \rar{\pi}\dar & X\dar\\
      \Spec \widehat{A}^\Gamma \rar & \Spec \widehat{A} \rar & \Spec A
    \end{tikzcd}
  \]
  where the squares are cartesian.
  By Theorem \ref{thm:gammaconst}, there exists a cofinite subset $\Gamma
  \subseteq \Lambda$ such that $\pi^\Gamma$ is a homeomorphism identifying
  $F$-pure loci.
  Since $X \times_A \widehat{A}^\Gamma$ is $F$-finite, the $F$-pure locus in
  $X \times_A \widehat{A}$ is therefore open by the fact that
  $(\ref{axiom:gammaopen})$ holds for $F$-purity (see the proof of Theorem
  \ref{thm:gammaconst}$(\ref{thm:gammaconstreg})$).
  \par Now let $x \in X \times_A \widehat{A}$.
  Since $A \to \widehat{A}$ is a regular homomorphism, the morphism $\pi$ is
  also regular by base change \cite[Prop.\ 6.8.3$(iii)$]{EGAIV2}.
  Thus, $\cO_{X \times_A \widehat{A},x}$ is $F$-pure if and only if
  $\cO_{X,\pi(x)}$ is $F$-pure by \cite[Prop.\ 5.13]{HR76} and \cite[Props.\
  2.4(4) and 2.4(6)]{Has10}.
  Denoting the $F$-pure locus in $X$ by $W$, we see that $\pi^{-1}(W)$ is
  the $F$-pure locus in $X \times_A \widehat{A}$.
  Since $\pi^{-1}(W)$ is open and $\pi$ is quasi-compact and faithfully flat by
  base change, the $F$-pure locus $W \subseteq X$ is open by \cite[Cor.\
  2.3.12]{EGAIV2}.
\end{proof}
\begin{remark}
  Although Lemma \ref{lem:finjopen} shows that the $F$-injective locus is open
  under $F$-finiteness hypotheses, and the gamma construction (Theorem
  \ref{thm:gammaconst}) implies that the $F$-injective locus is open for schemes
  essentially of finite type over \emph{complete} local rings,
  the fact that the $F$-injective locus is open under the hypotheses of
  Corollary \ref{cor:fpurelocusopen} is a recent result due to Rankeya Datta and
  the author \cite[Thm.\ B]{DM}.
\end{remark}
\subsection{Application to the minimal model program over imperfect fields}
\label{sect:tanakaapps}
With notation as in Construction \ref{constr:gamma},
let $\{X_i\}$ be a finite set of schemes essentially of finite type over $A$.
For each $i$, Theorem \ref{thm:gammaconst} produces a cofinite subset
$\Gamma_0^i \subseteq \Lambda$ such that properties of $X_i$ are inherited by
$X_i^\Gamma$ for every $\Gamma \subseteq \Gamma_0^i$ cofinite in $\Lambda$.
Setting $\Gamma_0 = \bigcap_i \Gamma_0^i$ gives a cofinite subset of $\Lambda$
which works for every scheme in the set $\{X_i\}$ at once.
We illustrate this strategy with the following:
\begin{corollary}\label{cor:gammaconst}
  Let $(X,\Delta)$ be a pair consisting of a normal variety $X$ over a field $k$
  of characteristic $p > 0$ and an $\RR$-Weil divisor $\Delta$ on $X$.
  Fix notation as in Construction \ref{constr:gamma}, where we set $A = k$.
  \begin{enumerate}[label=$(\roman*)$,ref=\roman*]
    \item If $X$ is a regular variety and $\Delta$ has simple normal crossing
      support, then there exists a cofinite subset $\Gamma_0 \subseteq \Lambda$
      such that $X^\Gamma$ is a regular variety and
      $(\pi^\Gamma)^*\Delta$ has simple normal crossing support for every
      cofinite subset $\Gamma \subseteq \Gamma_0$.\label{cor:gammaconstsnc}
    \item If $\dim X \le 3$ and $(X,\Delta)$ is klt (resp.\ log canonical), then
      there exists a cofinite subset $\Gamma_0 \subseteq \Lambda$
      such that $(X^\Gamma,(\pi^\Gamma)^*\Delta)$ is klt (resp.\ log
      canonical) for every cofinite subset $\Gamma \subseteq
      \Gamma_0$.\label{cor:gammaconstklt}
  \end{enumerate}
\end{corollary}
\begin{proof}
  For $(\ref{cor:gammaconstsnc})$, first write $\Delta = \sum a_iD_i$, where
  $a_i \in \RR$ and $D_i$ are prime divisors.
  By Theorem \ref{thm:gammaconst} applied to the regular locus of $X$ and of
  every set of intersections of the $D_i$, we see that there exists a cofinite
  subset $\Gamma_0 \subseteq \Lambda$ such that $X^\Gamma$ is a regular variety
  and $(\pi^\Gamma)^*\Delta = \sum a_i(\pi^\Gamma)^*D_i$ has simple normal
  crossing support for every $\Gamma \subseteq \Gamma_0$ cofinite in $\Lambda$.
  $(\ref{cor:gammaconstklt})$ then follows by applying
  $(\ref{cor:gammaconstsnc})$ to a log resolution of $(X,\Delta)$ while
  simultaneously choosing $\Gamma_0$ such that $X^\Gamma$ is normal for every
  cofinite subset $\Gamma \subseteq \Gamma_0$.
\end{proof}
\par Corollary \ref{cor:gammaconst}$(\ref{cor:gammaconstklt})$ easily provides
another method for proving the reduction step in \cite[Thm.\ 3.8]{Tan18}.
It can also be used to prove the more subtle reduction step in the
following result of Tanaka.
\begin{citedthm}[{\cite[Thm.\ 4.12]{Tan}}]
  Let $k$ be a field of characteristic $p > 0$.
  Let $(X,\Delta)$ be a log canonical surface over $k$, where $\Delta$ is a
  $\QQ$-Weil divisor.
  Let $f\colon X \to S$ be a projective morphism to a separated scheme $S$ of
  finite type over $k$.
  If $K_X+\Delta$ is $f$-nef, then $K_X+\Delta$ is $f$-semi-ample.
\end{citedthm}
\par The first step of the proof in \cite{Tan} is to reduce to the case where $k$ is
$F$-finite and contains an infinite perfect field in order to apply \cite[Thm.\
1]{Tan17}.
We illustrate how one can use the gamma construction (Theorem
\ref{thm:gammaconst}) to make this reduction.
\begin{proof}[Proof of reduction]
  Note that the formation of $K_X$ is compatible with ground field
  extensions \cite[Cor.\ V.3.4$(a)$]{Har66}, and that $f$-nefness is preserved
  under base change since $f$-ampleness is.
  By flat base change and the fact that field extensions are faithfully
  flat, $f$-semi-ampleness can be checked after a ground field extension.
  Since $k(x^{1/p^\infty})$ contains the infinite perfect field
  $\FF_p(x^{1/p^\infty})$ and applying the gamma construction to
  $k(x^{1/p^\infty})$ results in an $F$-finite field (Construction
  \ref{constr:gamma} and Lemma
  \ref{lem:gammaconstbasic}$(\ref{lem:gammaconstbasicffin})$), it therefore
  suffices to show that for some choice of $\Gamma$, the base change of
  $(X,\Delta)$ under the sequence of ground field extensions
  \[
    k \subseteq k(x^{1/p^\infty}) \subseteq
    \bigl(k(x^{1/p^\infty})\bigr)^{\Gamma}
  \]
  is a log canonical surface.
  Moreover, Corollary \ref{cor:gammaconst}$(\ref{cor:gammaconstklt})$ implies it
  suffices to prove that the base change of $(X,\Delta)$ to $k(x^{1/p^\infty})$
  is a log canonical surface.
  \par Fix a log resolution $\mu\colon Y \to X$ for $(X,\Delta)$, and write $K_Y
  - \mu^*(K_X+\Delta) = \sum_i a_iE_i$.
  Note that $k(x^{1/p^\infty}) = \bigcup_e k(x^{1/p^e})$, and that each
  field $k(x^{1/p^e})$ is isomorphic to $k(x)$.
  Since integrality, normality, and regularity are preserved under limits of
  schemes with affine and flat transition morphisms \cite[Cor.\ 5.13.4 and
  Prop.\ 5.13.7]{EGAIV2}, it
  suffices to show that $X \times_k k(x)$ is a normal variety, $Y
  \times_k k(x)$ is a regular variety, and each $E_i \times_k k(x)$ is a
  regular variety such that every intersection of the $E_i \times_k k(x)$'s is
  regular.
  This follows for $X \times_k k(x)$, since if $\bigcup_j U_j$ is an affine open
  covering of $X$, then $X \times_k k(x)$ is covered by affine open subsets that
  are localizations of the normal varieties $U_j \times_k \AA^1_k$, which
  pairwise intersect.
  A similar argument works for $Y$, the $E_i$'s, and the intersections of the
  $E_i$'s.
\end{proof}

\section{The ampleness criterion of de Fernex--K\"uronya--Lazarsfeld}
\label{sect:dfkl}
We now come to our first application of the gamma construction, Theorem
\ref{thm:dfkl41}.
Let $X$ be a projective variety of dimension $n > 0$.
For every Cartier divisor $L$ on $X$, we have
\[
  h^i\bigl(X,\cO_X(mL)\bigr) = O(m^n)
\]
for every $i$; see \cite[Ex.\ 1.2.20]{Laz04a}.
In \cite[Thm.\ 4.1]{dFKL07}, de Fernex, K\"uronya, and Lazarsfeld asked when the
higher cohomology groups have submaximal growth, i.e., when $h^i(X,\cO_X(mL)) =
o(m^n)$.
They proved that over the complex numbers, ample Cartier
divisors $L$ are characterized by having submaximal growth of higher cohomology
groups for small perturbations of $L$.
The content of Theorem \ref{thm:dfkl41} is that their characterization holds
for projective varieties over arbitrary fields.
Note that one can have $\widehat{h}^i(X,L) = 0$ for all $i > 0$ without $L$
being ample, or even pseudoeffective, hence the perturbation
by $A$ is necessary; see \cite[\S3.1]{Kur06} or \cite[Ex.\ 4.4]{ELMNP05}.
\par After reviewing some background material on asymptotic cohomological
functions in \S\ref{sect:burgosgil} following
\citeleft\citen{Kur06}\citemid
\S2\citepunct\citen{BGGJKM}\citemid \S3\citeright, we will prove an analogue of
a lemma on base loci \cite[Prop.\ 3.1]{dFKL07} in \S\ref{sect:lemonbaseloci}.
This latter subsection is where asymptotic test ideals are used.
Finally, we prove Theorem \ref{thm:dfkl41} in \S\ref{sect:dfkl41proof} using
the gamma construction and alterations.
\medskip
\par Before getting into the details of the proof, we briefly describe the main
difficulties in adapting the proof of \cite[Thm.\ 4.1]{dFKL07} to positive
characteristic.
First, the proof of \cite[Prop.\ 3.1]{dFKL07} requires resolutions of
singularities, and because of this, we can only prove a version of this lemma
(Proposition \ref{prop:dfkl31}) under the additional hypothesis that a specific
pair has (a weak version of) a log resolution.
This weaker result suffices for Theorem \ref{thm:dfkl41} since we can reduce to
this situation by taking the Stein factorization of an alteration.
Second, \cite{dFKL07} uses the assumption that the ground field is uncountable
to choose countably many very general divisors that facilitate an inductive
argument.
We reduce to the setting where the ground field is uncountable
by adjoining uncountably many indeterminates to our
ground field and then applying the gamma construction (Theorem
\ref{thm:gammaconstintro}) to reduce to the $F$-finite case; see Lemma
\ref{lem:dfklfieldred}.
\subsection{Background on asymptotic cohomological functions}
\label{sect:burgosgil}
We first review K\"uronya's asymptotic cohomological functions
with suitable modifications to work over arbitrary fields,
following \citeleft\citen{Kur06}\citemid
\S2\citepunct\citen{BGGJKM}\citemid \S3\citeright.
Asymptotic cohomological functions are defined as follows:
\begin{citeddef}[{\cite[Def.\ 3.4.6]{BGGJKM}}]\label{def:asymptoticcoh}
  Let $X$ be a projective scheme of dimension $n$ over a field.
  For every integer $i \ge 0$, the \textsl{$i$th asymptotic cohomological
  function} on $X$ is the function defined by setting
  \[
    \widehat{h}^i(X,D) \coloneqq \limsup_{m \to \infty}
    \frac{h^i\bigl(X,\cO_X\bigl(\lceil mD \rceil\bigr)\bigr)}{m^n/n!}
  \]
  for an $\RR$-Cartier divisor $D$ on $X$,
  where the round-up is defined by writing $D = \sum_i a_iD_i$ as an
  $\RR$-linear combination of Cartier divisors and setting $\lceil mD \rceil
  \coloneqq \sum_i \lceil ma_i \rceil D_i$; see \cite[Def.\ 3.4.1]{BGGJKM}.
  The numbers $\widehat{h}^i(X,D)$ only
  depend on the $\RR$-linear equivalence class of $D$ and are independent of
  the decomposition $D = \sum_i a_iD_i$ by \cite[Rem.\ 3.4.5]{BGGJKM}, hence
  $\widehat{h}^i(X,-)$ gives rise to well-defined functions $\Div_\RR(X) \to
  \RR$ and $\Div_\RR(X)/\mathord{\sim_\RR} \to \RR$.
\end{citeddef}
\par A key property of asymptotic cohomological functions is the following:
\begin{citedprop}[{\cite[Prop.\ 3.4.8]{BGGJKM}}]\label{prop:nahomogcont}
  Let $X$ be a projective scheme of dimension $n$ over a field.
  For every $i \ge 0$, the function $\widehat{h}^i(X,-)$ on
  $\Div_\RR(X)$ is homogeneous of degree $n$, and is
  continuous on every finite-dimensional $\RR$-subspace of
  $\Div_\RR(X)$ with respect to every norm.
\end{citedprop}
\par Proposition \ref{prop:nahomogcont} shows that Definition \ref{def:asymptoticcoh}
is equivalent to K\"uronya's original definition in \cite{Kur06}, and 
allows us to prove that asymptotic cohomological functions behave
well with respect to generically finite morphisms.
\begin{proposition}[cf.\ {\cite[Prop.\ 2.9(1)]{Kur06}}]\label{prop:kur291}
  Let $f\colon Y \to X$ be a surjective morphism of projective
  varieties, and consider an $\RR$-Cartier divisor $D$ on $X$.
  Suppose $f$ is generically finite of degree $d$.
  Then, for every $i$, we have
  \[
    \widehat{h}^i(Y,f^*D) = d \cdot \widehat{h}^i(X,D).
  \]
\end{proposition}
\begin{proof}
  The proof of \cite[Prop.\ 2.9(1)]{Kur06} works in our setting with the
  additional hypothesis that $D$ is a Cartier divisor.
  It therefore suffices to reduce to this case.
  If the statement holds for integral $D$, then it also
  holds for $D \in \Div_\QQ(X)$ by homogeneity of 
  $\widehat{h}^i$ (Proposition \ref{prop:nahomogcont}).
  Moreover, the subspace of $\Div_\RR(X)$ spanned by the Cartier divisors
  appearing in $D$ is finite-dimensional, hence by approximating each
  coefficient in $D$ by rational numbers, Proposition \ref{prop:nahomogcont}
  implies the statement for $D \in \Div_\RR(X)$ by continuity.
\end{proof}
\begin{remark}\label{rem:kur291reductions}
  We will repeatedly use the same steps as in the proof of Proposition
  \ref{prop:kur291} to prove statements about $\widehat{h}^i(X,D)$ for arbitrary
  $\RR$-Cartier divisors by reducing to the case when $D$ is a Cartier divisor.
  If $D$ is an $\RR$-Cartier divisor, we can write $D$ as the limit of
  $\QQ$-Cartier divisors by approximating each coefficient in a decomposition of
  $D$ by rational numbers, and continuity of asymptotic cohomological functions
  (Proposition \ref{prop:nahomogcont}) then allows us to reduce to the case when
  $D$ is a $\QQ$-Cartier divisor.
  By homogeneity of asymptotic cohomology functions (Proposition
  \ref{prop:nahomogcont}), one can then reduce to the case when $D$ is a Cartier
  divisor.
\end{remark}
\par We also need the following:
\begin{proposition}[Asymptotic Serre duality; cf.\ {\cite[Cor.\ 2.11]{Kur06}}]
  \label{prop:asympoticserre}
  Let $X$ be a projective variety of dimension $n$, and let $D$ be an
  $\RR$-Cartier divisor on $X$.
  Then, for every $0 \le i \le n$, we have
  \[
    \widehat{h}^i(X,D) = \widehat{h}^{n-i}(X,-D).
  \]
\end{proposition}
\begin{proof}
  By Remark \ref{rem:kur291reductions}, it suffices to consider the
  case when $D$ is integral.
  Let $f\colon Y \to X$ be a regular alteration of degree $d$ \cite[Thm.\
  4.1]{dJ96}.
  We then have
  \[
    \widehat{h}^i(Y,f^*D) = \limsup_{m \to \infty}
    \frac{h^{n-i}\bigl(Y,\cO_Y\bigl(K_Y-f^*(mD)\bigr)\bigr)}{m^n/n!} =
    \widehat{h}^{n-i}(Y,-f^*D)
  \]
  by Serre duality and \cite[Lem.\ 3.2.1]{BGGJKM}, respectively.
  By Proposition \ref{prop:kur291}, the left-hand side is equal to
  $d\cdot\widehat{h}^i(X,D)$ and the right-hand side is equal to
  $d\cdot\widehat{h}^{n-i}(X,-D)$, hence the statement follows after dividing by
  $d$.
\end{proof}

\subsection{A lemma on base loci}\label{sect:lemonbaseloci}
A key ingredient in our proof of Theorem \ref{thm:dfkl41} is the following
result on base loci, which is an analogue of \cite[Prop.\ 3.1]{dFKL07} over more
general fields.
In positive characteristic, we use asymptotic test ideals instead of
asymptotic multiplier ideals, which requires working over an $F$-finite
field.
\begin{proposition}\label{prop:dfkl31}
  Let $V$ be a normal projective variety of dimension at least two over
  an infinite field $k$, where if $\Char k = p > 0$, then we also assume that
  $k$ is $F$-finite.
  Let $D$ be a Cartier divisor on $V$.
  Assume there exists a closed subscheme $Z \subseteq V$ of pure dimension $1$
  such that
  \begin{enumerate}[label=$(\roman*)$]
    \item $D \cdot Z_\alpha < 0$ for every irreducible component $Z_\alpha$ of
      $Z$, and
    \item There exists a projective birational morphism $\mu\colon V' \to V$
      such that $V'$ is regular and $(\mu^{-1}(Z))_\red$ is a simple normal
      crossing divisor.
  \end{enumerate}
  Let $\fa \subseteq \cO_V$ be the ideal sheaf of $Z$.
  Then, there exist positive integers $q$ and $c$ such that for every integer $m
  \ge c$, we have
  \[
    \fb\bigl(\lvert mqD \rvert \bigr) \subseteq \fa^{m-c}.
  \]
\end{proposition}
\par Here, $\fb(\lvert D \rvert)$ denotes the \textsl{base ideal} of the Cartier
divisor $D$; see Definition \ref{def:baseideal}.
We note that the $Z_\alpha$ can possibly be non-reduced.
\par In the proof below, we will use the fact \cite[Lem.\ B.12]{Kle05} that if
$W$ is a one-dimensional subscheme of a complete scheme $X$ over a field, and if
$D$ is a Cartier divisor on $X$, then
\begin{equation}\label{eq:nonredintersection}
  (D \cdot W) = \sum_\alpha \length_{\cO_{X,\eta_\alpha}} \bigl(
  \cO_{W_\alpha,\eta_\alpha} \bigr) \cdot (D \cdot W_\alpha),
\end{equation}
where the $W_\alpha$ are the one-dimensional components of $W$ with generic
points $\eta_\alpha \in W_\alpha$.
\begin{proof}
  The statement is trivial if $H^0(V,\cO_V(mD)) = 0$ for every integer $m > 0$,
  since in this case $\fb(\lvert mqD \rvert) = 0$ for all positive integers
  $m,q$.
  We therefore assume $H^0(V,\cO_V(mD)) \ne 0$ for some integer $m > 0$.
  We will prove the statement in positive characteristic; see Remark
  \ref{rmk:dfkl31char0} for the characteristic zero case.
  \par We fix some notation.
  Set $D' = \mu^*D$ and set $E = (\mu^{-1}(Z))_\red$.
  We fix a very ample Cartier divisor $H$ on $V'$, and set $A =
  K_{V'} + (\dim V' + 1)H$.
  For every subvariety $W \subseteq V'$, a \textsl{complete intersection
  curve} is a curve formed by taking the intersection of $\dim W - 1$
  hyperplane sections in $\bigl\lvert H\rvert_W \bigr\rvert$, and a
  \textsl{general complete intersection curve} is one formed by taking these
  hyperplane sections to be general in $\bigl\lvert H\rvert_W \bigr\rvert$.
  For each positive integer $q$, we will consider the asymptotic test ideal
  \[
    \tau\bigl(V',\lVert qD' \rVert\bigr) = \tau\bigl(\lVert qD'
    \rVert\bigr) \subseteq \cO_{V'}.
  \]
  By uniform global generation for test ideals \cite[Prop.\
  4.1]{Sat18}, the sheaf
  \begin{equation}\label{eq:prop31gg}
    \tau\bigl(\lVert qD' \rVert\bigr) \otimes \cO_{V'}(qD'+A)
  \end{equation}
  is globally generated for every integer $q > 0$.
  \begin{step}\label{step:dfkl31step2}
    There exists an integer $\ell_0 > 0$ such that for every integer $\ell >
    \ell_0$ and for every irreducible component $F$ of $E$ that dominates
    $(Z_\alpha)_\red$ for some $\alpha$, we have
    \[
      \tau\bigl(\lVert \ell D' \rVert\bigr) \subseteq \cO_{V'}(-F).
    \]
  \end{step}
  \par Let $C \subseteq F$ be a general complete intersection curve; note that
  $C$ is integral by Bertini's theorem \cite[Thm.\ 3.4.10 and Cor.\
  3.4.14]{FOV99} and dominates $(Z_\alpha)_\red$ for some $\alpha$, hence $(D'
  \cdot C) < 0$ by the projection formula and \eqref{eq:nonredintersection}.
  If for some integer $q > 0$, the curve $C$ is not contained in the zero locus
  of $\tau(\lVert qD' \rVert)$, then the fact that the sheaf \eqref{eq:prop31gg}
  is globally generated implies
  \[
    \bigl( (qD' + A) \cdot C \bigr) \ge 0.
  \]
  Letting $\ell_{0F} = -(A \cdot C)/(D' \cdot C)$, we see that the ideal
  $\tau(\lVert \ell D' \rVert)$ vanishes everywhere along $C$ for every integer
  $\ell > \ell_{0F}$.
  By varying $C$, the ideal $\tau(\lVert \ell D' \rVert)$ must vanish everywhere
  along $F$ for every integer $\ell > \ell_{0F}$, hence we can set $\ell_0 =
  \max_F\{\ell_{0F}\}$.
  \begin{step}\label{step:dfkl31step3}
    Let $E_i$ be an irreducible component of $E$ not dominating
    $Z_\alpha$ for every $\alpha$.
    Suppose $E_j$ is another irreducible component of $E$ such that
    $E_i \cap E_j \ne \emptyset$ and for which there exists an integer $\ell_j$
    such that for every integer $\ell > \ell_j$, we have
    \begin{align*}
      \tau\bigl(\lVert \ell D' \rVert\bigr) &\subseteq \cO_{V'}(-E_j).
      \intertext{Then, there is an integer $\ell_i \ge \ell_j$ such that for
      every integer $\ell > \ell_i$, we have}
      \tau\bigl(\lVert \ell D' \rVert\bigr) &\subseteq \cO_{V'}(-E_i).
    \end{align*}
  \end{step}
  \par Let $C \subseteq E_i$ be a complete intersection curve.
  By the assumption that $E$ is a simple normal crossing divisor, there exists
  at least one closed point $P \in C \cap E_j$.
  For every $\ell > \ell_j$ and every $m > 0$, we have the sequence of
  inclusions
  \begin{equation}\label{eq:dfkl31incl}
    \begin{aligned}
      \MoveEqLeft[4]\Bigl( \tau\bigl(\lVert m\ell D'\rVert\bigr) \otimes
      \cO_{V'}(m\ell D'+A) \Bigr) \cdot \cO_C
      \subseteq \Bigl( \tau\bigl(\lVert \ell D'\rVert\bigr)^m \otimes
      \cO_{V'}(m\ell D'+A) \Bigr) \cdot \cO_C\\
      &\subseteq \Bigl( \cO_{V'}(-mE_j) \otimes
      \cO_{V'}(m\ell D'+A) \Bigr) \cdot \cO_C
      \subseteq \cO_C(A\rvert_C - mP)
    \end{aligned}
  \end{equation}
  where the first two inclusions follow from subadditivity \cite[Thm.\
  6.10(2)]{HY03} and by assumption, respectively.
  The last inclusion holds since $C$ maps to a closed point in $V$, hence
  $\cO_C(D') = \cO_C$.
  By the global generation of the sheaf in \eqref{eq:prop31gg} for $q = m\ell$,
  the inclusion \eqref{eq:dfkl31incl} implies that for every integer $\ell >
  \ell_j$, if $\tau(\lVert m\ell D' \rVert)$ does not vanish everywhere along
  $C$, then $(A \cdot C) \ge m$.
  Choosing $\ell_i = \ell_j \cdot ((A \cdot C) + 1 )$, we see that
  $\tau(\lVert \ell D' \rVert)$ vanishes everywhere along $C$ for every integer
  $\ell > \ell_i$.
  By varying $C$, we have $\tau(\lVert \ell D' \rVert) \subseteq
  \cO_{V'}(-E_i)$ for every integer $\ell > \ell_i$.
  \begin{step}\label{lem:dfkl32}
    There exists an integer $a > 0$ such that $\fb(\lvert maD' \rvert)
    \subseteq \cO_{V'}(-mE)$ for every integer $m > 0$.
  \end{step}
  \par Write
  \[
    E = \bigcup_j \bigcup_{i \in I_j} E_{ij},
  \]
  where the $E_{ij}$ are the irreducible components of $E$, and the
  $\bigcup_{i \in I_j} E_{ij}$ are the connected components of $E$.
  Since $V$ is normal, each preimage $\mu^{-1}(Z_\alpha)$ is connected by
  Zariski's main theorem \cite[Cor.\ III.11.4]{Har77}, hence each connected
  component $\bigcup_{i \in I_j} E_{ij}$ of $E$ contains an irreducible
  component $E_{i_0j}$ that dominates $(Z_\alpha)_\red$ for some $\alpha$.
  By Step \ref{step:dfkl31step2}, there exists an integer $\ell_0$
  such that for every $j$, we have $\tau(\lVert \ell D' \rVert) \subseteq
  \cO_{V'}(-E_{i_0j})$ for every integer $\ell > \ell_0$.
  For each $j$, by applying Step \ref{step:dfkl31step3} $(\lvert I_j
  \rvert-1)$ times to the $j$th connected component $\bigcup_{i \in I_j}
  E_{ij}$ of $E$, we can find $\ell_j$ such that
  $\tau(\lVert \ell D' \rVert) \subseteq \cO_{V'}(-E_{ij})$
  for every $i \in I_j$ and for every integer $\ell > \ell_j$.
  Setting $a = \max_j\{\ell_j\}+1$, we have $\tau(\lVert aD'
  \rVert) \subseteq \cO_{V'}(-E)$.
  Thus, for every integer $m > 0$, we have
  \[
    \fb\bigl(\lvert maD' \rvert\bigr) \subseteq \tau\bigl(\lvert
    maD'\rvert\bigr) \subseteq \tau\bigl(\lVert maD'\rVert\bigr) \subseteq
    \tau\bigl(\lVert aD'\rVert\bigr)^m \subseteq \cO_{V'}(-mE),
  \]
  where the first inclusion follows by the fact that $V'$ is regular hence
  strongly $F$-regular \cite[Props.\ 5.6(1) and 5.6(5)]{TW18}, the second
  inclusion is by definition of the asymptotic test ideal, and the third
  inclusion is by subadditivity \cite[Thm.\ 6.10(2)]{HY03}.
  \begin{step}
    Conclusion of proof of Proposition \ref{prop:dfkl31}.
  \end{step}
  \par Let $\pi\colon V'' \to V'$ be the normalized blowup of the
  ideal $\mu^{-1}\fa\cdot\cO_{V'}$, and write
  $(\mu\circ\pi)^{-1}\fa\cdot\cO_{V''} =
  \cO_{V''}(-E'')$ for a Cartier divisor $E''$ on
  $V''$.
  Note that since $\cO_{V''}(-(\pi^*E)_\red)$ is the radical of
  $\cO_{V''}(-E'')$, there exists an integer $b > 0$ such
  that $\cO_{V''}(-b(\pi^*E)_\red) \subseteq
  \cO_{V''}(-E'')$.
  We then have
  \[
    \fb\bigl(\lvert mab\,\pi^*D' \rvert\bigr) =
    \pi^{-1}\fb\bigl(\lvert mabD'\rvert\bigr) \cdot \cO_{V''}
    \subseteq \cO_{V''}(-mb\,\pi^*E) \subseteq
    \cO_{V''}\bigl(-mb(\pi^*E)_\red\bigr) \subseteq
    \cO_{V''}(-mE'')
  \]
  by Step \ref{lem:dfkl32}, where the first equality holds by
  Lemma \ref{lem:baselocusnormal}.
  Setting $q = ab$ and pushing forward by $\mu \circ \pi$, we have
  \[
    \fb\bigl(\lvert mqD \rvert\bigr) \subseteq (\mu \circ \pi)_*\fb\bigl(\lvert
    mq\,\pi^*D'\rvert\bigr) \subseteq (\mu \circ
    \pi)_*\cO_{V''}(-mE'') = \overline{\fa^m},
  \]
  where the first inclusion follows from Lemma \ref{lem:baselocusnormal}, and
  where $\overline{\fa^m}$ is the integral closure of $\fa^m$
  \cite[Rem.\ 9.6.4]{Laz04b}.
  Finally, given any ideal $\fa \subseteq \cO_V$, there exists an
  integer $c$ such that $\overline{\fa^{\ell+1}} = \fa \cdot
  \overline{\fa^\ell}$ for all $\ell \ge c$ \cite[Proof of Prop.\
  9.6.6]{Laz04b}, hence $\overline{\fa^m} \subseteq \fa^{m-c}$ for all $m \ge
  c$.
\end{proof}
\begin{remark}\label{rmk:dfkl31char0}
  When $\Char k = 0$, one can prove the stronger statement of
  \cite[Prop.\ 3.1]{dFKL07} using resolutions of singularities and the
  asymptotic multiplier ideals $\cJ(\lVert D \rVert)$ defined in \cite[Def.\
  11.1.2]{Laz04b} by replacing \cite[Prop.\ 5.6]{TW18} and \cite[Thm.\
  6.10(2)]{HY03} with \cite[Prop.\ 2.3]{dFM09} and \cite[Thm.\ A.2]{JM12},
  respectively.
  To replace \cite[Prop.\ 4.1]{Sat18}, one can pass to the algebraic closure
  (since the formation of multiplier ideals is compatible with ground field
  extensions \cite[Prop.\ 1.9]{JM12}) to deduce uniform global generation
  from the algebraically closed case \cite[Cor.\ 11.2.13]{Laz04b}.
\end{remark}
\par Before moving on to the proof of Theorem \ref{thm:dfkl41},
we note that after a preprint of this paper was posted, the authors of
\cite{MPST19} informed us that they had proved an asymptotic non-vanishing
statement \cite[Lem.\ 4.2]{MPST19} using techniques similar to ours in
Proposition \ref{prop:dfkl31} and Theorem \ref{thm:dfkl41}.
By combining the methods in this paper and in \cite[Lems.\ 4.3 and 4.4]{MPST19},
one can prove the full analogue of \cite[Prop.\ 3.1]{dFKL07}, namely:
\begin{proposition}\label{prop:dfkl31full}
  Let $V$ be a normal projective variety of dimension at least two over
  a field $k$.
  Let $D$ be a Cartier divisor on $V$, and suppose there exists an integral
  curve $Z \subseteq V$ such that $(D \cdot Z) < 0$.
  Denote by $\fa \subseteq \cO_V$ the ideal sheaf defining $Z$.
  Then, there exist positive integers $q$ and $c$ such that for every integer $m
  \ge c$, we have
  \[
    \fb\bigl(\lvert mqD \rvert \bigr) \subseteq \fa^{m-c}.
  \]
\end{proposition}
\par We will not use Proposition \ref{prop:dfkl31full} in the sequel.
See \cite[Prop.\ 6.2.1]{Mur19} for a proof.
\subsection{Proof of Theorem \ref{thm:dfkl41}}\label{sect:dfkl41proof}
We now prove Theorem \ref{thm:dfkl41}.
We first note that the direction $\Rightarrow$ in Theorem \ref{thm:dfkl41}
follows from existing results.
\begin{proof}[Proof of $\Rightarrow$ in Theorem \ref{thm:dfkl41}]
  Let $A$ be a very ample Cartier divisor.
  Then, for all $t$ such that $L-tA$ is ample, we have $\widehat{h}^i(X,L-tA) =
  0$ by Serre vanishing and by homogeneity and
  continuity; see Remark \ref{rem:kur291reductions}.
\end{proof}
\par For the direction $\Leftarrow$, it suffices to show Theorem \ref{thm:dfkl41} for
Cartier divisors $L$ by continuity and homogeneity; see Remark
\ref{rem:kur291reductions}.
We also make the following two reductions.
Recall that an $\RR$-Cartier divisor $L$ on $X$ \textsl{satisfies
$(\ref{thm:dfkl41cond})$} for a pair $(A,\varepsilon)$ consisting of a very
ample Cartier divisor $A$ on $X$ and a real number $\varepsilon > 0$ if
$\widehat{h}^i(X,L-tA) = 0$ for all $i > 0$ and all $t \in [0,\varepsilon)$.
\begin{lemma}\label{lem:dfklfieldred}
  To prove the direction $\Leftarrow$ in Theorem \ref{thm:dfkl41}, we may assume
  that the ground field $k$ is uncountable.
  In positive characteristic, we may also assume that $k$ is $F$-finite.
\end{lemma}
\begin{proof}
  We first construct a sequence
  \[
    k \subseteq k' \subseteq K
  \]
  of two field extensions such that $X \times_k K$ is integral, where $k'$ is
  uncountable and $K$ is $F$-finite in positive characteristic.
  If $k$ is already uncountable, then let $k' = k$.
  Otherwise, consider a purely transcendental extension
  \[
    k' \coloneqq k(x_\alpha)_{\alpha \in A}
  \]
  where $\{x_\alpha\}_{\alpha \in A}$ is an uncountable set of indeterminates;
  note that $k'$ is uncountable by construction.
  To show that $X \times_k k'$ is integral, let $\bigcup_j U_j$ be an affine
  open covering of $X$.
  Then, $X \times_k k'$ is covered by affine open subsets that are localizations
  of the integral varieties $U_j \times_k \Spec k[x_\alpha]_{\alpha \in A}$,
  which pairwise intersect, hence $X \times_k k'$ is integral.
  We set $K = k'$ in characteristic zero, and in positive characteristic, the
  gamma construction (Theorem \ref{thm:gammaconstintro}) shows that there is a
  field extension $k' \subseteq K$ such that $K$ is $F$-finite and the scheme $X
  \times_k K$ is integral.
  Note that $K$ is uncountable since it contains the uncountable field $k'$.
  \par Now suppose $X$ is a projective variety over $k$, and let $L$ be an
  Cartier divisor satisfying $(\ref{thm:dfkl41cond})$ for some pair
  $(A,\varepsilon)$.
  Let
  \[
    \pi\colon X \times_k K \longrightarrow X
  \]
  be the first projection map, which we note is faithfully flat by base change.
  Then, the pullback $\pi^*A$ of $A$ is very ample,
  and to show that $L$ is ample, it suffices to show that $\pi^*L$ is ample by
  flat base change and Serre's criterion for ampleness.
  By the special case of Theorem \ref{thm:dfkl41} over the ground field $K$, it
  therefore suffices to show that $\pi^*L$ satisfies $(\ref{thm:dfkl41cond})$
  for the pair $(\pi^*A,\varepsilon)$.
  \par We want to show that for every $i > 0$ and for all
  $t \in [0,\varepsilon)$, we have
  \begin{equation}\label{eq:hhatinvfieldext}
    \widehat{h}^i(X,L-tA) = \widehat{h}^i\bigl(X \times_k K,\pi^*(L-tA)\bigr) =
    0.
  \end{equation}
  For every $D \in \Div(X)$ and every $i \ge 0$, the number $h^i(X,\cO_X(D))$ is
  invariant under ground field extensions by flat base change,
  hence $\widehat{h}^i(X,D)$ is also.
  By homogeneity and continuity (see Remark \ref{rem:kur291reductions}), the
  number $\widehat{h}^i(X,D)$ is also invariant under ground field extensions
  for $D \in \Div_\RR(X)$, hence \eqref{eq:hhatinvfieldext} holds.
\end{proof}
\begin{remark}
  We note that if $k$ is $F$-finite or perfect, then one can construct a field
  extension $k \subseteq K$ as in Lemma \ref{lem:dfklfieldred} in a more
  elementary manner.
  When $k$ is $F$-finite of characteristic $p > 0$, then one can set $K$ to be
  $k(x^{1/p^\infty}_\alpha)_{\alpha \in A}$ for an uncountable set of
  indeterminates $\{x_\alpha\}_{\alpha \in A}$, since integrality and normality
  are preserved under limits of schemes with affine and flat transition
  morphisms \cite[Cor.\ 5.13.4]{EGAIV2}.
  When $k$ is perfect, then one can set $K$ to be a perfect closure of
  $k(x_\alpha)_{\alpha \in A}$.
  In this case, $X$ is geometrically reduced, and the morphism $X \times_k K \to
  X \times_k k(x_\alpha)_{\alpha \in A}$ is a homeomorphism since
  $k(x_\alpha)_{\alpha \in A} \subseteq K$ is purely inseparable \cite[Prop.\
  2.4.5$(i)$]{EGAIV2}.
  Thus, the base extension $X \times_k K$ is integral.
\end{remark}
\begin{lemma}\label{lem:nonnefcounterex}
  To prove the direction $\Leftarrow$ in Theorem \ref{thm:dfkl41}, it suffices
  to show that every Cartier divisor satisfying $(\ref{thm:dfkl41cond})$ is nef.
\end{lemma}
\begin{proof}
  Suppose $L$ is a Cartier divisor satisfying $(\ref{thm:dfkl41cond})$ for a
  pair $(A,\varepsilon)$.
  Choose $\delta \in (0,\varepsilon) \cap \QQ$ and let $m$ be a positive integer
  such that $m\delta$ is an integer.
  Then, the Cartier divisor $m(L-\delta A)$ is nef since
  \[
    \widehat{h}^i\bigl(X,m(L-\delta A) - tA\bigr)
    = \widehat{h}^i\bigl(X,mL - (t+m\delta)A\bigr)
    = m \cdot \widehat{h}^i\Bigl(X,L - \Bigl(\frac{t}{m}+\delta\Bigr)A\Bigr)
    = 0
  \]
  for all $t \in [0,m\varepsilon-\delta)$ by homogeneity (Proposition
  \ref{prop:nahomogcont}).
  Thus, the Cartier divisor $L = (L-\delta A) + \delta A$ is ample by
  \cite[Cor.\ 1.4.10]{Laz04a}.
\end{proof}
\par We will also need the following result to allow for an inductive proof.
Note that the proof in \cite{dFKL07} works in our setting.
\begin{citedlem}[{\cite[Lem.\ 4.3]{dFKL07}}]\label{lem:dfkl43}
  Let $X$ be a projective variety of dimension $n > 0$ over an uncountable
  field, and let $L$ be a Cartier divisor on $X$.
  Suppose $L$ satisfies $(\ref{thm:dfkl41cond})$ for a pair $(A,\varepsilon)$,
  and let $E \in \lvert A \rvert$ be a very general divisor.
  Then, the restriction $L\rvert_E$ satisfies $(\ref{thm:dfkl41cond})$ for the
  pair $(A\rvert_E,\varepsilon)$.
\end{citedlem}
\par We can now show the direction $\Leftarrow$ in Theorem \ref{thm:dfkl41}; by Lemma
\ref{lem:nonnefcounterex}, we need to show that every Cartier divisor satisfying
$(\ref{thm:dfkl41cond})$ is nef.
Recall that by Lemma \ref{lem:dfklfieldred}, we may assume that the ground field
$k$ is uncountable, and in positive characteristic, we may assume that $k$ is
$F$-finite as well.
Our proof follows that in \cite[pp.\ 450--454]{dFKL07} after reducing to a
setting where Proposition \ref{prop:dfkl31} applies, although we have to
be more careful in positive characteristic.
\begin{proof}[Proof of $\Leftarrow$ in Theorem \ref{thm:dfkl41}]
  We proceed by induction on $\dim X$.
  Suppose $\dim X = 1$; we will show the contrapositive.
  If $L$ is not nef, then $\deg L < 0$ and $-L$ is ample.
  Thus, by asymptotic Serre duality (Proposition \ref{prop:asympoticserre}), we
  have $\widehat{h}^1(X,L) = \widehat{h}^0(X,-L) \ne 0$,
  hence $(\ref{thm:dfkl41cond})$ does not hold for every choice of
  $(A,\varepsilon)$.
  \par We now assume $\dim X \ge 2$.
  Suppose by way of contradiction that there is a non-nef Cartier
  divisor $L$ satisfying $(\ref{thm:dfkl41cond})$.
  We first claim that 
  there exists a finite morphism $\nu\colon \widetilde{X} \to X$ such that
  $\nu^*L$ satisfies $(\ref{thm:dfkl41cond})$, and such that $\widetilde{X}$
  satisfies the hypotheses of Proposition \ref{prop:dfkl31} for
  $D = \nu^*L$.
  Choose an integral curve $Z \subset X$ such that $L \cdot Z < 0$, and let
  $\varphi\colon X' \to X$ be a regular alteration for the pair $(X,Z)$
  \cite[Thm.\ 4.1]{dJ96}, in which case $(\varphi^{-1}(Z))_\red$ is a
  simple normal crossing divisor.
  Consider the Stein factorization \cite[Cor.\ III.11.5]{Har77}
  \begin{equation}\label{eq:step2factor}
    \begin{tikzcd}
      X' \rar{\mu}\arrow{dr}[swap]{\varphi} & \widetilde{X}\dar{\nu}\\
      & X
    \end{tikzcd}
  \end{equation}
  for the morphism $\varphi$, in which case $\widetilde{X}$ is a normal
  projective variety.
  Now let $\widetilde{Z}$ be the scheme-theoretic inverse image of $Z$ under
  $\nu$, and write
  \[
    \widetilde{Z} = \bigcup_\alpha \widetilde{Z}_\alpha
  \]
  where $\widetilde{Z}_\alpha$ are the irreducible components of
  $\widetilde{Z}$.
  Since $\nu$ is finite, every $\widetilde{Z}_\alpha$ is one-dimensional and
  dominates $Z$, hence the projection formula and \eqref{eq:nonredintersection}
  imply $\nu^*L \cdot \widetilde{Z}_\alpha < 0$.
  Finally, $(\varphi^{-1}(Z))_\red = (\mu^{-1}(\widetilde{Z}))_\red$ is a simple
  normal crossing divisor by the factorization \eqref{eq:step2factor}.
  \par We now show that $\nu^*L$ satisfies $(\ref{thm:dfkl41cond})$.
  Since $\nu^*A$ is ample \cite[Prop.\ 1.2.13]{Laz04a}, we can choose a
  positive integer $a$ such that $a\,\nu^*A$ is very ample.
  Then, Proposition \ref{prop:kur291} implies
  \begin{equation}\label{eq:pullbacklsatisfiesstar}
    \widehat{h}^i(\widetilde{X},\nu^*L-ta\,\nu^*A)
    = (\deg\nu) \cdot \widehat{h}^i(X,L-taA) = 0
  \end{equation}
  for all $i > 0$ and for all $t \in [0,\varepsilon/a)$.
  Replacing $A$ by $aA$, we will assume that $\nu^*A$ is very ample.
  \medskip
  \par For the rest of the proof, our goal is to show that
  \begin{equation}\label{eq:dfklgoal}
    \widehat{h}^1(\widetilde{X},\nu^*L-\delta\nu^*A) \ne 0
  \end{equation}
  for $0 < \delta \ll 1$, contradicting \eqref{eq:pullbacklsatisfiesstar}.
  Let $F \in \lvert \nu^*A \rvert$ be a very general divisor.
  By Bertini's theorem \cite[Thm.\ 3.4.10 and Cor.\ 3.4.14]{FOV99}, we may
  assume that $F$ is a subvariety of $\widetilde{X}$, in which case by
  inductive hypothesis and Lemma \ref{lem:dfkl43}, we have that $\nu^*L\rvert_F$
  is ample.
  Since ampleness is an open condition in families \cite[Cor.\ 9.6.4]{EGAIV3},
  there exists an integer $b > 0$ such that $b\,\nu^*L$ is very ample along the
  generic divisor $F_\eta \in \lvert \nu^*A \rvert$.
  By possibly replacing $b$ with a multiple, we may also assume that
  $mb\,\nu^*L\rvert_{F_\eta}$ has vanishing higher cohomology for every integer
  $m > 0$.
  Since the ground field $k$ is uncountable, we can then choose a sequence of
  very general Cartier divisors $\{E_\beta\}_{\beta=1}^\infty \subseteq \lvert
  \nu^*A \rvert$ such that the following properties hold:
  \begin{enumerate}[label=$(\alph*)$]
    \item $E_\beta$ is a subvariety of $\widetilde{X}$ for all $\beta$
      (by Bertini's theorem \cite[Thm.\ 3.4.10 and Cor.\ 3.4.14]{FOV99});
    \item For all $\beta$, $b\,\nu^*L\rvert_{E_\beta}$ is very ample
      and $mb\,\nu^*L\rvert_{E_\beta}$ has vanishing higher cohomology for every
      integer $m > 0$ (by the constructibility of very ampleness in families
      \cite[Prop.\ 9.6.3]{EGAIV3} and by semicontinuity); and
    \item For every positive integer $r$ and for all non-negative integers
      $j$ and $m$, the $k$-dimension of cohomology groups
      of the form
      \begin{equation}\label{eq:indepcohgp}
        H^j\bigl(E_{\beta_1} \cap E_{\beta_2} \cap \cdots \cap
          E_{\beta_r},\cO_{E_{\beta_1} \cap E_{\beta_2} \cap \cdots \cap
        E_{\beta_r}}(mL)\bigr)
      \end{equation}
      is independent of the $r$-tuple $(\beta_1,\beta_2,\ldots,\beta_r)$
      (by semicontinuity; see \cite[Prop.\
      5.5]{Kur06}).
  \end{enumerate}
  We will denote by $h^j(\cO_{E_1 \cap E_2 \cap \cdots \cap E_r}(mL))$ the
  dimensions of the cohomology groups \eqref{eq:indepcohgp}.
  By homogeneity (Proposition \ref{prop:nahomogcont}), we can
  replace $L$ by $bL$ so that $\nu^*L\rvert_{E_\beta}$ is very ample
  with vanishing higher cohomology for all $\beta$.
  \par To show \eqref{eq:dfklgoal}, we now follow the proof in
  \cite[pp.\ 453--454]{dFKL07} with appropriate modifications.
  Given positive integers $m$ and $r$, consider the complex
  \begin{align*}
    K^\bullet_{m,r} \coloneqq{}& \biggl(
    \bigotimes_{\beta=1}^r\bigl(\cO_{\widetilde{X}} \longrightarrow
    \cO_{E_\beta}\bigr) \biggr) \otimes \cO_{\widetilde{X}}(m\,\nu^*L)\\
    ={}& \biggl\{ \cO_{\widetilde{X}}(m\,\nu^*L) \longrightarrow
    \bigoplus_{\beta=1}^r \cO_{E_\beta}(m\,\nu^*L) \longrightarrow \bigoplus_{1
    \le \beta_1 < \beta_2 \le r} \cO_{E_{\beta_1} \cap E_{\beta_2}}(m\,\nu^*L)
    \longrightarrow \cdots \biggr\}.
  \end{align*}
  By \cite[Cor.\ 4.2]{Kur06}, this complex is acyclic away from
  $\cO_{\widetilde{X}}(m\,\nu^*L)$, hence is a resolution for
  $\cO_{\widetilde{X}}(m\,\nu^*L - r\,\nu^*A)$.
  In particular, we have
  \[
    H^j\bigl(\widetilde{X},\cO_{\widetilde{X}}(m\,\nu^*L - r\,\nu^*A)\bigr) =
    \HH^j(\widetilde{X},K^\bullet_{m,r}).
  \]
  The right-hand side is computed by an $E_1$-spectral sequence whose first page
  is
  \[
    \begin{tikzpicture}
      \matrix (m) [matrix of math nodes,commutative diagrams/every cell,column
      sep=1.8em,transform shape,nodes={scale=0.95}]{
        \strut\makebox[1em][c]{$\mathllap{E_1}$} &[-1em] &
        \hphantom{\bigoplus\limits_{\beta=1}^r
        H^0\bigl(\cO_{E_\beta}(m\,\nu^*L)\bigr)}
        & &[-1.75em]
        &[-1.75em] {}\\[-0.75em]
        \vdots & \vdots & & & \\
        2 & \vphantom{\bigoplus\limits^r}
        H^2\bigl(\cO_{\widetilde{X}}(m\,\nu^*L)\bigr)\\
        1 & \vphantom{\bigoplus\limits^r}
        H^1\bigl(\cO_{\widetilde{X}}(m\,\nu^*L)\bigr) &
        & & & \vphantom{\bigoplus\limits^r}\hphantom{\cdots}\\
        0 & H^0\bigl(\cO_{\widetilde{X}}(m\,\nu^*L)\bigr)
        \vphantom{\bigoplus\limits_{\beta=1}^r} & \bigoplus\limits_{\beta=1}^r
        H^0\bigl(\cO_{E_\beta}(m\,\nu^*L)\bigr) & \bigoplus\limits_{1
        \le \beta_1 < \beta_2 \le r} H^0\bigl(\cO_{E_{\beta_1} \cap
        E_{\beta_2}}(m\,\nu^*L)\bigr) & \cdots
        &\vphantom{\bigoplus\limits_{\beta=1}^r}\\
        \quad\strut & 0 & 1 & 2 & \cdots&\strut \\};
      \draw[<-] (m-1-1.east) node[anchor=south]{$q$} -- (m-6-1.east);
      \draw[->] (m-6-1.north) -- (m-6-6.north) node[anchor=west]{$p$};
      \draw (m-1-3.west) |- (m-4-6.south west);
      \draw[draw=none,fill=gray,fill opacity=0.2] (m-1-3.west) rectangle
        node[midway,opacity=1] {\Large $0$} (m-4-6.south west);
      \path[commutative diagrams/.cd, every arrow, every label]
        (m-5-2) edge node {$v_{m,r}$} (m-5-3)
        (m-5-3) edge node {$u_{m,r}$} (m-5-4);
    \end{tikzpicture}
  \]
  hence there is a natural inclusion
  \begin{equation}\label{eq:dfkl41incl}
    \frac{\ker(u_{m,r})}{\im(v_{m,r})} \subseteq
    H^1\bigl(\widetilde{X},\cO_{\widetilde{X}}(m\,\nu^*L - r\,\nu^*A)\bigr).
  \end{equation}
  \par We want to bound the left-hand side of \eqref{eq:dfkl41incl} from
  below.
  First, there exists a constant $C_1 > 0$ such that $h^0(\cO_{E_1\cap
  E_2}(m\,\nu^*L)) \le C_1 \cdot m^{n-2}$ for all $m \gg 0$ \cite[Ex.\
  1.2.20]{Laz04a}.
  Thus, we have
  \[
    \codim\Bigl(\ker(u_{m,r}) \subseteq \bigoplus\limits_{\beta=1}^r
    H^0\bigl(E_\beta,\cO_{E_\beta}(m\,\nu^*L)\bigr) \Bigr)
    \le C_2 \cdot r^2m^{n-2}
  \]
  for some $C_2$ and for all $m \gg 0$.
  Now by Proposition \ref{prop:dfkl31}, there are positive integers $q$ and
  $c$ such that $\fb(\lvert mq\,\nu^*L \rvert) \subseteq \fa^{m-c}$
  for all $m > c$, where $\fa$ is the ideal sheaf of $\widetilde{Z}$.
  By replacing $L$ by $qL$, we can assume that this
  inclusion holds for $q = 1$.
  The morphism $v_{m,r}$ therefore fits into the following commutative diagram:
  \[
    \begin{tikzcd}
      H^0\bigl(\widetilde{X},\cO_{\widetilde{X}}(m\,\nu^*L) \otimes
      \fa^{m-c}\bigr) \rar{v'_{m,r}} \dar[equals]
      & \operatorname{\smash{\displaystyle\bigoplus\limits_{\beta=1}^r}}
      H^0\bigl(E_\beta,\cO_{E_\beta}(m\,\nu^*L) \otimes
      \fa^{m-c}\bigr) \dar[hook]\\
      H^0\bigl(\widetilde{X},\cO_{\widetilde{X}}(m\,\nu^*L)\bigr)
      \rar{v_{m,r}} &
      \operatorname{\smash{\displaystyle\bigoplus\limits_{\beta=1}^r}}
      H^0\bigl(E_\beta,\cO_{E_\beta}(m\,\nu^*L)\bigr)\\[-1em]
    \end{tikzcd}
  \]
  \par We claim that there exists a constant $C_3 > 0$ such that for all $m \gg
  0$,
  \begin{equation}\label{eq:dfkl15}
    \codim\Bigl(H^0\bigl(E_\beta,\cO_{E_\beta}(m\,\nu^*L) \otimes
      \fa^{m-c}\bigr) \subseteq H^0\bigl(E_\beta,\cO_{E_\beta}(m\,\nu^*L)\bigr)
    \Bigr) \ge C_3 \cdot m^{n-1}.
  \end{equation}
  Granted this, we have
  \[
    \dim\biggl( \frac{\ker(u_{m,r})}{\im(v_{m,r})} \biggr) \ge C_4 \cdot \bigl(
    rm^{n-1} - r^2m^{n-2}\bigr)
  \]
  for some constant $C_4 > 0$ and for all $m \gg 0$.
  Fixing a rational number $0 < \delta \ll 1$ and setting $r = m\delta$ for an
  integer $m > 0$ such that $m\delta$ is an integer, we then see that
  there exists a constant $C_5 > 0$ such that
  \[
    h^1\bigl(\widetilde{X},\cO_{\widetilde{X}}\bigl(m(\nu^*L - \delta
    \,\nu^*A) \bigr)\bigr) \ge C_5 \cdot \delta m^n
  \]
  for all $m \gg 0$, contradicting \eqref{eq:pullbacklsatisfiesstar}.
  \par It remains to show \eqref{eq:dfkl15}.
  Since the vanishing locus of $\fa$ may have no $k$-rational points, we will
  pass to the algebraic closure of $k$ to bound the codimension on
  the left-hand side of \eqref{eq:dfkl15} from below.
  Let $\overline{E}_\beta \coloneqq E_\beta \times_k \overline{k}$, and denote
  by $\pi\colon \overline{E}_\beta \to E_\beta$ the projection
  morphism.
  Note that
  \begin{align*}
    \MoveEqLeft[2] \codim\Bigl(H^0\bigl(E_\beta,\cO_{E_\beta}(m\,\nu^*L)
    \otimes \fa^{m-c}\bigr) \subseteq
  H^0\bigl(E_\beta,\cO_{E_\beta}(m\,\nu^*L)\bigr) \Bigr)\\
    &= \codim\Bigl(H^0\bigl(\overline{E}_\beta,\cO_{\overline{E}_\beta}(m\,
      \pi^*\nu^*L) \otimes \pi^{-1}\fa^{m-c} \cdot \cO_{\overline{E}_\beta}
      \bigr) \subseteq H^0\bigl(\overline{E}_\beta,\cO_{\overline{E}_\beta}(m\,
      \pi^*\nu^*L)\bigr)\Bigr)
  \intertext{by the flatness of $k \subseteq \overline{k}$.
  Since $\cO_{\overline{E}_\beta}(\pi^*\nu^*L)$ is very ample by base change, we
  can choose a closed point $x \in Z(\pi^{-1}\fa \cdot \cO_{\overline{E}_\beta})
  \cap \overline{E}_\beta$, in which case
  $m\,\pi^*\nu^*L$ separates $(m-c)$-jets at $x$ by \cite[Proof of Lem.\
  3.7]{Ito13}.
  Finally, the dimension of the space of $(m-c)$-jets at $x$ is at least that
  for a regular point of a variety of dimension $n$, hence}
    \MoveEqLeft[2] \codim\Bigl(H^0\bigl(\overline{E}_\beta,
    \cO_{\overline{E}_\beta}(m\,\pi^*\nu^*L) \otimes \pi^{-1}\fa^{m-c} \cdot
    \cO_{\overline{E}_\beta} \bigr) \subseteq
    H^0\bigl(\overline{E}_\beta,\cO_{\overline{E}_\beta}(m\,\pi^*\nu^*L)\bigr)
    \Bigr)\\
    &\ge \codim\Bigl(H^0\bigl(\overline{E}_\beta,\cO_{\overline{E}_\beta}(m\,
      \pi^*\nu^*L) \otimes \fm_x^{m-c} \cdot \cO_{\overline{E}_\beta} \bigr)
      \subseteq H^0\bigl(\overline{E}_\beta,\cO_{\overline{E}_\beta}(m\,
      \pi^*\nu^*L)\bigr)\Bigr)\\
      &\ge \binom{m-c+n}{n-1} \ge C_3 \cdot m^{n-1}
  \end{align*}
  for some constant $C_3 > 0$ and all $m \gg 0$, as required.
\end{proof}

\section{Nakayama's theorem on restricted base loci}\label{sect:nakayamabminus}
We now come to our second application of the gamma construction, Theorem
\ref{thm:bbpconj}.
This result extends known cases of the following conjecture due to Boucksom,
Broustet, and Pacienza.
\begin{citedconj}[{\cite[Conj.\ 2.7]{BBP13}}]
  Let $X$ be a normal projective variety, and let $D$ be a pseudoeffective
  $\RR$-divisor on $X$.
  Then, we have $\Bminus(D) = \NNef(D)$.
\end{citedconj}
\par In \S\ref{sect:bminus}, we define the restricted base locus $\Bminus(D)$
and the non-nef locus $\NNef(D)$.
We then prove Theorem \ref{thm:bbpconj} in \S\ref{sect:proofbbp} using the
gamma construction (Theorem \ref{thm:gammaconst}) to reduce to the $F$-finite
case, in which case it suffices to apply results in \cite{Sat18}.
We recall that $\lVert \cdot \rVert$ denotes a compatible choice of Euclidean
norms on the vector spaces $N^1_\QQ(X)$ and $N^1_\RR(X)$, which are finite
dimensional for complete schemes $X$ by \cite[Prop.\ 2.3]{Cut15}.
\subsection{Background on restricted base loci and non-nef loci}
\label{sect:bminus}
We start by defining the following numerically invariant approximation of the
stable base locus defined in Definition \ref{def:stablebaselocus}.
\begin{citeddef}[{\cite[Def.\ 1.12]{ELMNP06}}]
  Let $X$ be a projective scheme over a field, and let $D$ be an
  $\RR$-Cartier divisor on $X$.
  The \textsl{restricted base locus} of $D$ is the subset
  \[
    \Bminus(D) \coloneqq \bigcup_A \SB(D+A)
  \]
  of $X$, where the union runs over all ample $\RR$-Cartier
  divisors $A$ such that $D + A$ is a $\QQ$-Cartier divisor.
  Note that $\Bminus(D) = \emptyset$ if and only if $D$ is nef \cite[Ex.\
  1.18]{ELMNP06}.
\end{citeddef}
\par We will need the following result, which says that the formation of restricted
base loci is compatible with ground field extensions.
\begin{lemma}\label{lem:bminusfieldext}
  Let $X$ be a projective scheme over a field $k$, and let $D$ be an
  $\RR$-Cartier divisor on $X$.
  Let $k \subseteq k'$ be a field extension with corresponding projection
  morphism $\pi\colon X \times_k k' \to X$.
  Then,
  \[
    \Bminus(\pi^*D) = \pi^{-1}\bigl(\Bminus(D)\bigr).
  \]
\end{lemma}
\begin{proof}
  Let $\{A_n\}_{n \ge 1}$ be a sequence of ample $\RR$-Cartier divisors such
  that $\lim_{n\to\infty}\lVert A_n \rVert = 0$ and such that $D+A_n$ is a
  $\QQ$-Cartier divisor for every $n$.
  By \cite[Prop.\ 1.19]{ELMNP06}, we have
  \[
    \Bminus(D) = \bigcup_{n \ge 1} \SB(D+A_n).
  \]
  By flat base change, we have $\pi^{-1}(\SB(D+A_n)) = \SB(\pi^*(D+A_n))$, hence
  \[
    \pi^{-1}\bigl(\Bminus(D)\bigr) = \bigcup_{n \ge 1}
    \SB\bigl(\pi^*(D+A_n)\bigr) = \Bminus(\pi^*D),
  \]
  where the second equality follows from applying \cite[Prop.\ 1.19]{ELMNP06}
  again to the sequence $\{\pi^*A_n\}_{n \ge 1}$ of ample $\RR$-Cartier divisors
  on $X \times_k k'$.
\end{proof}
\par Next, we want to define the non-nef locus.
\begin{citeddef}[{\citeleft\citen{Nak04}\citemid Def.\
  III.2.2\citepunct\citen{CDB13}\citemid Def.\ 2.11\citeright}]
  Let $X$ be a normal projective variety, and let $D$ be a big $\RR$-Cartier
  divisor.
  Consider a divisorial valuation $v$ on $X$.
  The \textsl{numerical vanishing order} of $D$ along $v$ is
  \[
    v_{\num}(D) \coloneqq \inf_{E \equiv_\RR D} v(E),
  \]
  where the infimum runs over all effective $\RR$-Cartier divisors
  $\RR$-numerically equivalent to $D$.
  When $D$ is a pseudoeffective $\RR$-Cartier divisor, we set
  \[
    v_{\num}(D) \coloneqq \sup_A v_{\num}(D + A),
  \]
  where the supremum runs over all ample $\RR$-Cartier divisors $A$ on $X$, and
  where we note that $D+A$ is a big $\RR$-Cartier divisor by \cite[Thm.\
  2.2.26]{Laz04a}.
  The \textsl{non-nef locus} of a pseudoeffective $\RR$-Cartier divisor $D$
  is
  \[
    \NNef(D) \coloneqq \bigcup_v c_X(v)
  \]
  where the union runs over all divisorial valuations such that $v_\num(D) > 0$,
  and $c_X(v)$ is the center of the divisorial valuation $v$.
  Note that $\NNef(D) = \emptyset$ if and only if $D$ is nef \cite[Rem.\
  III.2.8]{Nak04}.
\end{citeddef}
\par To prove Theorem \ref{thm:bbpconj}, we will also use the following:
\begin{citeddef}[{\cite[Def.\ 2.2 and Rem.\ 2.3]{ELMNP06}}]
  Let $X$ be a normal projective variety, and let $D$ be a $\QQ$-Cartier
  divisor.
  Consider a divisorial valuation $v$ on $X$.
  The \textsl{asymptotic order of vanishing} of $D$ along $v$ is
  \[
    v\bigl(\lVert D \rVert\bigr) \coloneqq \inf_{E \sim_\QQ D} v(E)
  \]
  where the infimum runs over all effective $\QQ$-Cartier divisors
  $\QQ$-linearly equivalent to $D$.
\end{citeddef}
\subsection{Proof of Theorem \ref{thm:bbpconj}}\label{sect:proofbbp}
We start by proving a version of Theorem \ref{thm:bbpconj} for arbitrary normal
projective varieties.
We fix some notation.
Let $X$ be a normal projective variety over a field $k$.
If $\Char k = 0$ and $\Delta$ is an effective $\QQ$-Weil divisor on $X$ such
that $K_X+\Delta$ is $\QQ$-Cartier, then the \textsl{non-klt locus} of the pair
$(X,\Delta)$ is $\Nklt(X,\Delta) \coloneqq Z(\cJ(X,\Delta))$, where
$\cJ(X,\Delta)$ is the multiplier ideal \cite[Def.\ 9.3.56]{Laz04b}, and the
\textsl{non-klt locus} of $X$ is
\[
  \Nklt(X) \coloneqq \bigcap_{\Delta} \Nklt(X,\Delta),
\]
where the intersection runs over all effective $\QQ$-Weil divisors $\Delta$ such that
$K_X+\Delta$ is $\QQ$-Cartier.
If $\Char k = p > 0$, then the \textsl{non-strongly $F$-regular locus} of $X$
is
\[
  \NSFR(X) \coloneqq \bigl\{x \in X \bigm\vert \cO_{X,x}\ \text{is not strongly
  $F$-regular}\bigr\}.
\]
If $X$ is $F$-finite, then $\NSFR(X) = Z(\tau(X))$, since test ideals localize
\cite[Prop.\ 3.23$(ii)$]{Sch11}, and since $\tau(R) = R$ for an $F$-finite ring
$R$ if and only if $R$ is strongly $F$-regular \cite[Prop.\ 5.6(5)]{TW18}.
\begin{theorem}[cf.\ {\citeleft\citen{CDB13}\citemid Cor.\
  4.7\citepunct\citen{Sat18}\citemid Cor.\ 4.7\citeright}]
  \label{thm:bbpconjwloci}
  Let $X$ be a normal projective variety over a field $k$, and let $D$ be a
  pseudoeffective $\RR$-Cartier divisor on $X$.
  If $\Char k = 0$, then
  \begin{align}
    \Bminus(D) \smallsetminus \Nklt(X) &= \NNef(D) \smallsetminus \Nklt(X),
    \label{eq:bbpconjwlocichar0}
    \intertext{and if $\Char k = p > 0$, then}
    \Bminus(D) \smallsetminus \NSFR(X) &= \NNef(D) \smallsetminus \NSFR(X).
    \label{eq:bbpconjwlocicharp}
  \end{align}
\end{theorem}
\begin{proof}
  We first prove that $\NNef(D) \subseteq \Bminus(D)$ for every pseudoeffective
  $\RR$-Cartier divisor $D$, following \cite[Lem.\ 2.6]{BBP13}.
  Let $x \notin \Bminus(D)$, and let $v$ be a divisorial valuation such that $x
  \in c_X(v)$.
  By definition of $\Bminus(D)$, there exists an ample $\RR$-Cartier divisor $A$
  such that $D+A$ is $\QQ$-Cartier divisor for which $x \notin \SB(D+A)$.
  Thus, there exists an effective $\QQ$-Cartier divisor $E$ such that $E
  \sim_\QQ D+A$ and such that $x \notin \Supp E$.
  We therefore have $v_\num(D+A) \le v(E) = 0$.
  \par It remains to show the inclusions $\subseteq$.
  We first consider the case when $\Char k = 0$.
  Let $\Delta$ be an effective $\QQ$-Weil divisor such that $K_X+\Delta$ is
  $\QQ$-Cartier.
  The proof of \cite[Thm.\ 4.5]{CDB13} holds in this setting after replacing the
  application of Nadel vanishing and Castelnuovo--Mumford regularity in the
  proof of \cite[Lem.\ 4.1]{CDB13} with the uniform global generation
  statement mentioned in Remark \ref{rmk:dfkl31char0},
  hence $\Bminus(D) \smallsetminus \Nklt(X,\Delta) \subseteq \NNef(D)
  \smallsetminus \Nklt(X,\Delta)$.
  Taking the union over all $\QQ$-Weil divisors $\Delta$ such that $K_X+\Delta$
  is $\QQ$-Cartier, we then see that the inclusion $\subseteq$ holds in
  \eqref{eq:bbpconjwlocichar0}.
  \par We now consider the characteristic $p > 0$ case.
  By \cite[Lems.\ 2.12 and 2.13]{CDB13}, there exists
  a sequence $\{A_n\}_{n \ge 0}$ of ample $\RR$-Cartier divisors on $X$
  such that $D+A_n$ is a $\QQ$-Cartier divisor for every $n$,
  $\lim_{n\to \infty}\lVert A_n \rVert \to 0$, and
  \[
    \Bminus(D) = \bigcup_n \Bminus(D+A_n) \quad \text{and} \quad
    \NNef(D) = \bigcup_n \NNef(D+A_n).
  \]
  By proving the inclusion $\subseteq$ in \eqref{eq:bbpconjwlocicharp} for
  $D+A_n$, it therefore suffices to consider the case when $D$ is a big
  $\QQ$-Cartier divisor.
  Let $x \in \Bminus(D)$, and consider a divisorial valuation $v$ on $X$ such
  that $c_X(v) = \overline{\{x\}}$, which is given by the order of vanishing
  along a prime Cartier divisor $E$ on a normal birational model $X'$ of $X$.
  By applying the gamma construction (Theorem \ref{thm:gammaconst}) to $X$,
  $X'$, and $E$, there exists a field extension $k \subseteq k^\Gamma$ such that
  $X \times_k k^\Gamma$ and $X' \times_k k^\Gamma$ are normal varieties, $E
  \times_k k^\Gamma$ is a prime divisor, and $\pi^\Gamma(\NSFR(X^\Gamma)) =
  \NSFR(X)$.
  Note that the order of vanishing along $E \times_k k^\Gamma$ defines a
  divisorial valuation $v^\Gamma$ on $X^\Gamma$ extending $v$.
  Since $\Bminus((\pi^\Gamma)^*D) = (\pi^\Gamma)^{-1}(\Bminus(D))$ by Lemma
  \ref{lem:bminusfieldext}, we have $(\pi^\Gamma)^{-1}(x) \in
  \Bminus((\pi^\Gamma)^*D)$, hence \cite[Cor.\ 4.6]{Sat18} implies
  \[
    (\pi^\Gamma)^{-1}(x) \in \bigcup_{m \ge 1} Z\Bigl(\tau\bigl(X^\Gamma,m\cdot
    \bigl\lVert (\pi^\Gamma)^*D\bigr\rVert\bigr)\Bigr).
  \]
  By the proof of the implication $(2) \Rightarrow (5)$ in \cite[Prop.\
  3.17]{Sat18} (which does not use the assumption that $K_X$ is $\QQ$-Cartier),
  we see that $v^\Gamma(\lVert (\pi^\Gamma)^*D \rVert) > 0$, and by pulling back
  Cartier divisors in $\lvert D \rvert$ to $X^\Gamma$, we have $v(\lVert D
  \rVert) > 0$ as well.
  Finally, since $D$ is big, \cite[Lem.\ 3.3]{ELMNP06} implies $v_\num(D) =
  v(\lVert D \rVert) > 0$, hence $x \in \NNef(D)$, and the inclusion
  $\subseteq$ holds in \eqref{eq:bbpconjwlocicharp}.
\end{proof}
\par We now prove Theorem \ref{thm:bbpconj}.
\begin{proof}[Proof of Theorem \ref{thm:bbpconj}]
  As in the proof of Theorem \ref{thm:bbpconjwloci}, the inclusion $\NNef(D)
  \subseteq \Bminus(D)$ holds, hence it suffices to show the reverse inclusion.
  By Theorem \ref{thm:bbpconjwloci}, we have $\Bminus(D) \smallsetminus \Nklt(X)
  \subseteq \NNef(D)$ (resp.\ $\Bminus(D) \smallsetminus \NSFR(X)
  \subseteq \NNef(D)$).
  Now let $\{A_n\}_{n \ge 1}$ be a sequence of ample
  $\RR$-Cartier divisors such that $\lim_{n\to \infty}\lVert A_n \rVert \to 0$
  and such that $D+A_n$ is a $\QQ$-Cartier divisor for every $n$.
  By \cite[Prop.\ 1.19]{ELMNP06}, we have
  \[
    \Bminus(D) = \bigcup_{n \ge 1} \SB(D+A_n).
  \]
  Since each $\SB(D+A_n)$ does not contain any isolated points \cite[Prop.\
  1.1]{ELMNP09}, we see that $\Bminus(D)$ does not contain any isolated points.
  Finally, since $\Nklt(X)$ (resp.\ $\NSFR(X)$) is a discrete set of isolated
  closed points by assumption, we have $\Bminus(D) \subseteq \NNef(D)$
  by Theorem \ref{thm:bbpconjwloci}.
\end{proof}

\appendix
\section{Some results on \textit{F}-injective rings}\label{app:finj}
Let $R$ be a noetherian ring of prime characteristic $p > 0$.
Recall from Definition \ref{def:finj} that $R$ is \textsl{$F$-injective} if for
every maximal ideal $\fm \subseteq R$, the $R$-module homomorphism
$H^i_\fm(F) \colon H^i_\fm(R_\fm) \to H^i_\fm(F_*R_\fm)$ induced by Frobenius is
injective for all $i$.
\par In this appendix, we prove some facts about $F$-injective rings for which
we could not find a reference.
First, we characterize $F$-finite rings that are $F$-injective using
Grothendieck duality.
This characterization is already implicit in \cite[Rem.\ on p.\ 473]{Fed83} and
the proof of \cite[Prop.\ 4.3]{Sch09}.
Note that if $R$ is an $F$-finite ring, then the exceptional pullback $F^!$ from
Grothendieck duality exists \cite[III.6]{Har66}, and $R$ has a normalized
dualizing complex $\omega_R^\bullet$ by Theorem \ref{thm:ffiniteaffine}.
\begin{lemma}[cf.\ {\cite[Rem.\ on p.\ 473]{Fed83}}]\label{lem:finjffin}
  Let $R$ be an $F$-finite noetherian ring of prime characteristic $p > 0$.
  Then, $R$ is $F$-injective if and only if the $R$-module homomorphisms
  \begin{equation}\label{eq:grottracefinj}
    \mathbf{h}^{-i}\Tr_F\colon \mathbf{h}^{-i}F_*F^!\omega_R^\bullet
    \longrightarrow \mathbf{h}^{-i}\omega_R^\bullet
  \end{equation}
  induced by the Grothendieck trace of Frobenius are surjective for all $i$.
\end{lemma}
\par Lemma \ref{lem:finjffin} is most useful when $R$ is essentially of finite type
over an $F$-finite field, in which case
$F^!\omega_R^\bullet \simeq \omega_R^\bullet$ in the derived category
$\DD^+_{\qc}(R)$ \cite[Thm.\ 5.3]{Nay09}, hence the homomorphisms in
\eqref{eq:grottracefinj} can be written as $\mathbf{h}^{-i}F_*\omega_R^\bullet
\to \mathbf{h}^{-i}\omega_R^\bullet$.
\begin{proof}
  By Grothendieck local duality \cite[Cor.\ V.6.3]{Har66}, $R$ is $F$-injective
  if and only if
  \begin{align*}
    F^*\colon \Ext^{-i}_R(F_*R,\omega_R^\bullet) &\longrightarrow
    \Ext^{-i}_R(R,\omega_R^\bullet)
    \intertext{is surjective for all $i$.
    By Grothendieck duality for finite morphisms \cite[Thm.\ III.6.7]{Har66},
    this occurs if and only if}
    F_*\Ext^{-i}_R(R,F^!\omega_R^\bullet) &\longrightarrow
    \Ext^{-i}_R(R,\omega_R^\bullet)
  \end{align*}
  is surjective for all $i$.
  Since $\Ext_R^{-i}(R,-) = \mathbf{h}^{-i}(-)$ and by the description
  of the Grothendieck duality isomorphism \cite[Thm.\ III.6.7]{Har66}, this
  is equivalent to the surjectivity of \eqref{eq:grottracefinj} for all $i$.
\end{proof}
\par Next, we prove that $F$-injectivity is an open condition on $F$-finite schemes.
This extends \cite[Prop.\ 3.12]{QS17} to the non-local case.
\begin{lemma}[cf.\ {\citeleft\citen{QS17}\citemid Prop.\ 3.12\citepunct
  \citen{Sch09}\citemid Prop.\ 4.3\citeright}]
  \label{lem:finjopen}
  If $R$ is an $F$-finite noetherian ring of prime characteristic $p > 0$, then the
  locus
  $\{\fp \in \Spec R \mid R_\fp\ \text{is $F$-injective}\}$
  is open.
  In particular, $R$ is $F$-injective if and only if $R_\fp$ is $F$-injective
  for every prime ideal $\fp \subseteq R$.
\end{lemma}
\begin{proof}
  For each integer $i$, let $M_i$ be the cokernel of the $R$-module homomorphism
  in \eqref{eq:grottracefinj}.
  Since the Grothendieck trace is compatible with flat base change \cite[Prop.\
  III.6.6(2)]{Har66} and Frobenius is compatible with localizations, the support
  of $M_i$ is the locus where $R$ is not $F$-injective by Lemma
  \ref{lem:finjffin}.
  Now $\omega_R^\bullet$ and $F_*F^!\omega_R^\bullet$ have coherent
  cohomology that is nonzero in only finitely many degrees by definition of
  $\omega_R^\bullet$ and \cite[Prop.\ III.6.1]{Har66}, hence each $M_i$ is
  finitely generated over $R$, has closed support, and is nonzero for only
  finitely many $i$.
  The locus
  \begin{equation}\label{eq:finjlocusassupp}
    \Spec R \smallsetminus \biggl(\bigcup_{i} \Supp M_i\biggr) =
    \bigl\{\fp \in \Spec R \bigm\vert R_\fp\ \text{is $F$-injective}\bigr\}.
  \end{equation}
  is therefore open, proving the first statement.
  The locus \eqref{eq:finjlocusassupp} is in particular closed under
  generization, hence $R$ is $F$-injective only if $R_\fp$ is $F$-injective for
  every prime ideal $\fp \subseteq R$.
  The converse implication holds by definition, proving the second statement.
\end{proof}
\par Finally, to prove Theorem \ref{thm:gammaconst}, we used the following descent
property for $F$-injectivity.
\begin{lemma}[cf.\ {\cite[Lem.\ 4.6]{Has10}}]\label{lem:fsingsflatextfinj}
  Let $\varphi\colon R \to S$ be a pure homomorphism of
  rings of prime characteristic $p > 0$.
  If $I \subseteq R$ is a finitely generated ideal such that the Frobenius
  action
  \begin{alignat*}{3}
    H^i_{IS}(F_S)&\colon& H^i_{IS}(S) &\longrightarrow H^i_{IS}(F_{S*}S)
    \intertext{is injective for some $i$, then the Frobenius action}
    H^i_I(F_R)&\colon& H^i_I(R) &\longrightarrow H^i_I(F_{R*}R)
  \end{alignat*}
  is also injective.
  \par In particular, suppose that $(R,\fm)$ and $(S,\fn)$ are noetherian local
  rings and that $\varphi$ is a local homomorphism with zero-dimensional closed
  fiber.
  If $S$ is $F$-injective, then $R$ is $F$-injective.
\end{lemma}
\par In the non-noetherian case, we define local cohomology modules by
taking injective resolutions in the category of sheaves of abelian groups on
spectra, as is done in \cite[Exp.\ I, Def.\ 2.1]{SGA2}.
\begin{proof}
  For each $i$, we have the following commutative square:
  \begin{equation}\label{eq:finjflatext}
    \begin{tikzcd}[column sep=large]
      H_I^i(R) \rar{H_I^i(F_R)}\dar[hook] & H_I^i(F_{R*}R) \dar\\
      H_{IS}^i(S) \rar[hook]{H_{IS}^i(F_S)} & H_{IS}^i(F_{S*}S)
    \end{tikzcd}
  \end{equation}
  Since $\varphi$ is pure, the map $H^i_I(R) \hookrightarrow H^i_{IS}(S)$ is
  injective by applying \cite[Cor.\ 6.6]{HR74} to the \v{C}ech complex, which
  computes local cohomology by \cite[Exp.\ II, Prop.\ 5]{SGA2}.
  The bottom horizontal arrow in \eqref{eq:finjflatext} is injective by
  assumption, hence the top horizontal arrow $H^i_I(F_R)$ in
  \eqref{eq:finjflatext} is also injective by the commutativity of the diagram.
  The second statement in the lemma is a special case of the first,
  since under the given assumptions, we have $\sqrt{\fm S} = \fn$, hence
  $H^i_{\fm S}(S) \simeq H^i_\fn(S)$.
\end{proof}
\section{Strong \textit{F}-regularity for non-\textit{F}-finite rings}
\label{app:nonffinfreg}
In this appendix, we describe the relationship between different notions of
strong $F$-regularity in the non-$F$-finite setting.
This material is based on \cite[\S3]{Has10} and \cite[\S6]{DS16}.
\begin{definition}
  \label{def:fregapp}
  Let $R$ be a noetherian ring of prime characteristic $p > 0$.
  We follow the notation in Definition \ref{def:freg}.
  For every $c \in R$, we also say that $R$ is \textsl{$F$-split along $c$}
  if $\lambda^e_c$ splits as an $R$-module homomorphism for some $e > 0$.
  We then say that
  \begin{enumerate}[label=$(\alph*)$,ref=\alph*]
    \item $R$ is \textsl{split $F$-regular} if $R$ is $F$-split along every $c
      \in R^\circ$ \cite[Def.\ 5.1]{HH94};\label{def:fsingssplitreg}
    \item $R$ is \textsl{$F$-pure regular} if $R$ is $F$-pure along every $c \in
      R^\circ$ \cite[Rem.\ 5.3]{HH94}; and\label{def:fsingspureregds}
    \item $R$ is \textsl{strongly $F$-regular} if every inclusion of $R$-modules
      is tightly closed \cite[Def.\ 3.3]{Has10}.
      \label{def:fsingspureregapp}
  \end{enumerate}
  The definition in $(\ref{def:fsingspureregapp})$ is due to Hochster; see
  \cite[Def.\ 8.2]{HH90} for the definition of tight closure for modules.
  While $(\ref{def:fsingspureregapp})$ is not the definition used in the rest of
  this paper (Definition \ref{def:freg}$(\ref{def:fsingspurereg})$), these
  two definitions are equivalent by \cite[Lem.\ 3.6]{Has10}.
  \par Note that $(\ref{def:fsingssplitreg})$ is the usual definition of strong
  $F$-regularity in the $F$-finite setting.
  The terminology in $(\ref{def:fsingssplitreg})$ and
  $(\ref{def:fsingspureregds})$ is from \cite[Defs.\ 6.6.1 and 6.1.1]{DS16}.
  $F$-pure regular rings are called \textsl{very strongly $F$-regular} in
  \cite[Def.\ 3.4]{Has10}.
\end{definition}
\par The relationship between these notions of strong $F$-regularity can be
summarized as follows:
\[
  \begin{tikzcd}[column sep=7.5em,row sep=large]
    \text{$F$-split regular} \rar[Rightarrow]{\text{split maps are pure}}
    & \text{$F$-pure regular}
    \dar[Rightarrow,swap]{\text{\cite[Lem.\ 3.8]{Has10}}}\\
    & \text{strongly $F$-regular} \uar[Rightarrow,dashed,bend
    right=30,swap]{\substack{\text{local}\\\text{\cite[Lem.\ 3.6]{Has10}}}}
    \arrow[start anchor=west,Rightarrow,dashed,bend left=17]{ul}[description,
    pos=0.55]{\substack{\text{$F$-finite}\\\text{\cite[Lem.\ 3.9]{Has10}}}}
    & \text{regular}
    \arrow[Rightarrow]{l}[swap]{\text{\cite[Thm.\ 6.2.1]{DS16}}}
  \end{tikzcd}
\]

\end{document}